\documentclass[11pt]{article}
\usepackage{times,dsfont,ifthen,mathrsfs,latexsym,amsfonts,color}
\usepackage{amsmath,amsthm,makeidx,fontenc,amssymb,bm,graphicx,psfrag,listings,curves,extarrows,enumitem}
\usepackage[
linktocpage=true,colorlinks,citecolor=magenta,linkcolor=blue,urlcolor=magenta]{hyperref}

\usepackage{cite}
\usepackage{geometry}
\usepackage{indentfirst}
\usepackage{amssymb}
\makeatletter

\newcommand{\Rmnum}[1]{\expandafter\@slowromancap\romannumeral #1@}
\makeatother

\usepackage[showonlyrefs]{mathtools}  
\mathtoolsset{showonlyrefs=true}

\newtheorem{theorem}{Theorem}[section]
\newtheorem{lemma}{Lemma}[section]

\newtheorem{proposition}{Proposition}[section]
\newtheorem{remark}{Remark}[section]

\newcommand{\R}{{\mathbb R}}
\newcommand{\RN}{{\mathbb R^n}}  
\newcommand{\bean}{\begin{eqnarray*}}
	\newcommand{\eean}{\end{eqnarray*}}

\newcommand{\sbr}[1]{\left(#1\right)}
\newcommand{\mbr}[1]{\left[#1\right]}
\newcommand{\lbr}[1]{\left\{#1\right\}}

\newcommand{\abs}[1]{\left\lvert#1\right\rvert}

\newcommand{\dx}{ ~\mathrm{d} x}
\newcommand{\nm}[1]{\Vert #1 \Vert}

\setitemize{itemindent=38pt,leftmargin=0pt,itemsep=-0.4ex,listparindent=26pt,partopsep=0pt,parsep=0.5ex,topsep=-0.25ex}

\numberwithin{equation}{section}

\begin{document}
	\theoremstyle{plain}

	\title{\bf    Uniqueness of bound states to the logarithmic  Schr\"odinger  equation
	}  
	
	\date{}
	
	\author{Tianhao Liu\thanks{ School of Mathematics, Statistics and Mechanics, Beijing University of Technology, Beijing 100124, China.
			Email: liuthmath@gmail.com},\;\;		
		Xin Sun\thanks{Department of Mathematical Sciences, Tsinghua University, Beijing 100084, China. 
			Email: Sun\_Xin2001@outlook.com},\;\;         
		Wenming Zou \thanks{Department of Mathematical Sciences, Tsinghua University, Beijing 100084, China.  Email: zou-wm@mail.tsinghua.edu.cn}
	}

	\maketitle
	\begin{center}
		\begin{minipage}{140mm}
			\begin{center}{\bf Abstract}\end{center}	\small 
			This paper studies the uniqueness of   bound states for the problem
			\begin{equation} 
				\Delta u + u\log  u ^2=0, \quad  u\in  H^1(\RN), \quad  n\geq 2,
			\end{equation}
			which arises from the logarithmic  Schr\"odinger   equation. We prove that for every  integer  $k\geq 1$, there exists  a  unique   radial solution $u(r)=u(|x|)$  that has exactly $k$ simple zeros for $r>0$.
			This resolves an open problem posed by Troy [{Arch. Ration. Mech. Anal.} 222 (2016), 1581--1600] and confirms the Berestycki-Lions conjecture  for the logarithmic nonlinearity. The proof combines the shooting method with suitable auxiliary functions introduced by Tang [{Invent. math.}  243 (2026), 245--291]. A major difficulty arises from the singular behavior of   the nonlinearity $f(u)=u \log u^2$ at origin. We overcome it by establishing  asymptotic convergence and sharp decay rates  at infinity for  any ground state or bound state. More precisely, every such solution satisfies
			\begin{equation}
				\lim_{r\to\infty}\frac{  u'(r)}{u(r) \sqrt{\abs{ \log u^2(r)}}}=\lim_{r\to\infty}\frac{u'(r)}{ru(r)} = -1, \quad \limsup_{r\to\infty}|u(r)|e^{(\frac12-\epsilon)r^2}<\infty, ~~\forall \epsilon \in ( 0,\frac{1}{2}).
			\end{equation}
			These asymptotic behaviors are of independent interest and may be useful for other problems involving logarithmic  nonlinearities. 
			\vskip0.13in
			
			
			\vskip0.23in					
		\end{minipage}
	\end{center}

	\section{Introduction}\label{SectionIntroduction}
	In this paper, we investigate  the uniqueness of  solutions for the following problem  \begin{equation} \label{equlog} 
		\Delta u + u\log  u ^2=0, \quad  u\in  H^1(\RN), \quad n\geq 1.
	\end{equation}
	A positive solution  $u\in H^1(\RN)$ of \eqref{equlog}  is called a ground state, while a sign-changing radial solution  $u\in H^1(\RN)$ is referred to as a bound state. It   is known from \cite{Bialynicki1976,Bialynicki1979} that there exists a ground state which is radially symmetric and positive in $\RN$, the so-called Gausson:
	\begin{equation}\label{Gausson}
		u_0(x)= \exp\sbr{-\frac{|x|^2}{2}+\frac{n}
			{2}}  , \quad  n\geq 1.
	\end{equation}

	\vskip 0.04in

	The  Gausson \eqref{Gausson} plays a central role in applications of the   time-dependent   logarithmic Schr\"{o}dinger   equation
	\begin{equation} \label{E2}
		i\frac{ \partial  \psi}{\partial t} =\Delta \psi  +\psi \log|\psi|^2, \quad  ~~\psi=\psi(x,t)\in\mathbb{C},  \quad  ~~(x,t)\in\RN\times\R^+,
	\end{equation} 
	which was    first introduced  in the seminal work of Bialynicki-Birula and Mycielski  \cite{Bialynicki1976} to describe   the separability of non-interacting subsystems in quantum mechanics. Since then,  it has attracted  considerable attention due to its applications in diverse areas of physics,   such as quantum mechanics, quantum optics, transport and diffusion
	phenomena, information theory, quantum gravity, and the theory
	of Bose-Einstein condensates (see \cite{TroyARMA2018}   and the references therein).  Moreover, as shown in \cite{TroyARMA2018},   equation \eqref{equlog} can be formally derived both   as the limit $p\to 1^+$  of the classical reaction-diffusion equation 
	\begin{equation}
		\frac{ \partial  u}{\partial t} =\Delta u+ |u|^{p-1}u- u,
	\end{equation}
	and also from the nonlinear Klein-Gordon equation
	\begin{equation}
		\frac{ \partial^2  u}{\partial t^2}=\Delta u + |u|^{p-1}u- u.
	\end{equation}

	\vskip 0.04in
	
	It is well known that, after a suitable scaling, the ground state of \eqref{equlog} serves as an extremal function for the logarithmic Sobolev inequality (see \cite[Theorem 8.14]{lieb-loss}):
	\begin{equation}\label{log Sobolev ineq}
		\int_{\RN} u^2 \log u^2 \dx\leq \frac{a^2}{\pi} \int_{\RN} |\nabla u| ^2   \dx+\mbr{\log\|u\|_{L^2(\RN)} ^2 -n(1+\log a)} \|u\|_{L^2(\RN)} ^2,
	\end{equation}
	which holds for any   function  $ u\in H ^1(\RN) $ with  $n\geq 2$  and any constant $ a>0$. In this respect, the uniqueness of ground states is essential  not only for the equation \eqref{equlog} itself but also for the sharpness of the logarithmic Sobolev inequality. A classical version involving the Gauss measure was first studied by Stam \cite{Stam} and further by Gross \cite{Gross-AJM1975}; we also refer to \cite{pino-jmpa2002,Beckner-1992,Brigati-jfa2024} for related developments.

	\vskip 0.12in

	The uniqueness of the ground state for $n\geq 1$ has been extensively investigated. For $n\geq 3$,   d'Avenia, Montefusco, and Squassina  \cite{avenia2014}  used the moving plane method \cite{GNN} to prove that any positive solution of \eqref{equlog} is radially symmetric about some point.  Combining this  with the uniqueness result of Serrin and Tang  \cite{SerrinTang2000}, they  concluded that positive solution $u_0(x)$   is the unique   ground state of \eqref{equlog} for  $n\geq 3$.    Independently, Troy \cite{TroyARMA2018} developed a novel comparison method, which combines  $u_0(r)$
	with energy estimates and associated Ricatti equation estimates, to establish the uniqueness  of ground state for    $2\leq n\leq 9$. The situation  $n=1$ is straightforward, because equation \eqref{equlog} is one-dimensional and then has the first integral 
	\begin{equation}
		\frac{{u'}^2 }{2}+  \frac{1}{2} u^2\sbr{\log  u ^2-1} =c,  
	\end{equation}
	where $c$ is a constant.  Since $u\in H^1(\R)$, we infer  from the asymptotic condition  $	(u(r),u'(r)) \to (0,0) $ as $ r\to \infty $ that  $c=0$. Then   it is straightforward to verify that $u_0$ is the unique ground state to \eqref{radialform}.

	\vskip 0.08in
	The existence of infinitely many  bound states of \eqref{equlog} was first established by d'Avenia et al. \cite{avenia2014} using variational methods.    For the  non-autonomous case we refer to \cite{alves-ji-cvpde2020,ji-szulkin2016,ss-2015,tanaka-zhang-cvpde,WangZhangZhang-ADE2024,Wang=ARMA=2019,LiuWeiZou2025,W.ShuaiNonlinearity2019} and the references therein.  In particular, by using the Nehari's ODE technique \cite{Nehari},  Shuai \cite{W.ShuaiNonlinearity2019}  proved  that for  any  $n\geq 2$,   problem  \eqref{equlog}   admits   sign-changing radial solutions $u(r)= u(|x|)$ with any prescribed number of zeros for $r>0$. These solutions are  high-energy bound states,  and satisfies the  initial value problem
	\begin{equation}\label{radialform}
		\begin{aligned} 
			&	u'' +\frac{n-1}{r}u'+ u \log u^2=0,  \quad r>0, \quad  n\geq 2,
			\\ & u(0)=\alpha>0, \quad 	u'(0)=0,
		\end{aligned}			
	\end{equation}
	with  the asymptotic condition
	\begin{equation}\label{decayassum}
		(u(r),u'(r)) \to (0,0)\quad \text{ as } r\to \infty,
	\end{equation}
	where throughout this paper,    $'$ denotes differentiation with respect to $r$. 
	
	\vskip 0.08in
	In \cite[Section 3]{TroyARMA2018}, Troy proposed the following open problem, which remains unresolved:
	\vskip 0.12in
	\noindent\textbf{Troy's open problem \cite{TroyARMA2018}:} \textit{When $n>1$ are sign-changing solutions of \eqref{radialform}--\eqref{decayassum} with prescribed numbers of zeros unique?  }

	\vskip 0.12in
	This uniqueness problem  is part of a broader investigation into radial solutions of semilinear elliptic equations. Indeed, over the past few decades, the   uniqueness of radial solutions to semilinear elliptic equations of the form
	\begin{equation}\label{equgeneral}
		\Delta u + f(u) = 0, \quad u \in H^1(\mathbb{R}^n),
	\end{equation}
	has been intensively studied  for a wide class of nonlinearities $f$.  A classical result by Gidas, Ni and Nirenberg \cite{GNN} shows that under suitable assumptions on $f$, all positive  ground states are radially symmetric. The existence problem for ground states and bound states has been extensively treated for various nonlinearities in the literature; see, for instance, \cite{MTW,Cortazar2015,Berestycki-1,Berestycki-2,JK,S1977,Ryder,Nehari1963}. 
	The study of the uniqueness of the ground state was initiated by Coffman \cite{Coffman1972}, who proved the uniqueness of the ground state for the  equation 
	\begin{equation}\label{equcubic}
		\Delta u -u+u^3=0, \quad u \in H^1(\mathbb{R}^3).
	\end{equation}  This line of research was subsequently extended by McLeod and Serrin \cite{MS1981,MS1987} to more general nonlinearities, including the following power-type case,  
	\begin{equation}\label{equpower}
		\Delta u -u+|u|^{p-1}u=0, \quad u \in H^1(\mathbb{R}^n),
	\end{equation}
	where they proved the uniqueness  of the ground state for  the following parameter ranges:
	\begin{equation}
		1<p<\infty, ~ \text{ for } ~  n=1,2; \quad  1<p\leq {n}/(n-2), ~ \text{ for } ~ n=3,4;  \quad  1<p\leq  {8}/{n}, ~ \text{ for } ~ 4<n<8.
	\end{equation}
	Subsequently, Kwong \cite{Kwong1989}   proved the uniqueness of   ground state solutions of  \eqref{equpower}  for all    admissible parameters, namely  
	\begin{equation}\label{ass1}
		1<p<\infty, ~ \text{ for } ~  n=1,2, \quad  \quad   1<p< (n+2)/(n-2), ~\text{ for } ~ n\geq3.  
	\end{equation}
	For the case $p \geq (n+2)/(n-2)$ with $n\geq 3$, the  Pohozaev identity   \cite{pohozaev1965} implies that \eqref{equpower} admits no nontrivial solutions.   Since a comprehensive review of the literature is beyond the scope of this work, we refer the reader to \cite{FrankSummer,SerrinTang2000,Kwong-Zhang-DIE1991,CL1991,CEF1998,CJ1993,FL2013,FLS2016,KL1992,M1993,OS1999,Y1991,ET1997,FLS1996,PS1983,PS1998,Tao2006,N1983} and the references therein for further results on the uniqueness of the ground state. We also refer to \cite{PS1983,NN1985,DDG2013} for results on the non‑uniqueness of ground states.

	\vskip 0.1in
	In contrast, the uniqueness of sign-changing bound states is less understood, and the related literature remains limited. We define a \textit{$k$-node bound state}  $u\in H^1(\RN)$ as a nodal (i.e., sign-changing) radial  solution of  \eqref{equgeneral} with exactly $k$ zeros. In the celebrated works of Berestycki and Lions \cite{Berestycki-1,Berestycki-2}, the authors proposed a conjecture on the
	characterization of sign-changing bound states.
	
	\vskip 0.12in
	\noindent\textbf{ The conjecture of Berestycki and Lions \cite{Berestycki-1,Berestycki-2}:} \textit{We conjecture that, at least for some classes of $f$'s, there is one and exactly one solution of  \eqref{equgeneral} that has precisely $k- 1$ nodes (i.e., simple zeroes)  for $r\in \sbr{0,\infty}$. }  
	
	\vskip 0.12in
	
	This conjecture remains open for most nonlinearities, but significant progress has been achieved on several specific models \cite{Cohen2024,Troy2005proc,AoWeiYao2016,TangInvent2026,Cortazar2011,Cortazar2009}. The intricate influence of  $f$ on the solution structure is often subtle and nontrivial. For the cubic nonlinearity $f(u)=-u+u^3$,   Cohen, Li and Schlag \cite{Cohen2024} proved the uniqueness of the first twenty bound states to \eqref{equcubic}  via rigorous numerical analysis, thereby making substantial progress on a long-standing  open problem raised by Hastings and McLeod \cite{Hastings2012}. For the  power-type nonlinearity  $f(u)=-u+|u|^{p-1}u$,    Ao, Wei and Yao \cite{AoWeiYao2016} proved that the 1-node bound state  to \eqref{equpower} is unique when the exponent $p$ is sufficiently close to the critical exponent $(n+2)/(n-2)$.    In a recent breakthrough, Tang \cite{TangInvent2026} completely resolved the Berestycki--Lions conjecture for the power-type nonlinearity satisfying \eqref{ass1} when $n\geq 3$, proving the uniqueness of bound states to \eqref{equpower} with any prescribed number of zeros. On the other hand, for the piecewise linear nonlinearity, 
	\begin{equation}\label{piecewise}
		f(u)=\begin{cases}
			u+1,  &\quad u\leq -1/2,\\
			-u,    &\quad  u\in \sbr{-1/2,1/2},\\
			u-1,    &\quad  u\geq 1/2,\\
		\end{cases}
	\end{equation}
	Troy  \cite{Troy2005proc}  proved the existence and uniqueness of the  1-node bound state  for $n=3$.  The uniqueness of $k$-node bound states for arbitrary $k\in\mathbb{N}$ was resolved by  Cort\'azar,   Garc\'ia-Huidobro and  Yarur \cite{Cortazar2011}.  In fact, the authors of \cite{Cortazar2011} considered 
	the more general equation \eqref{equgeneral} under the following 
	assumptions on $f$:
	\begin{itemize}[ 
		itemindent = 0pt,
		labelindent = \parindent,
		labelwidth = 2em,
		labelsep = 5pt,
		leftmargin = *]\vskip0.06in
		\item[($f_1$)]  $f$ is odd, $f (0) = 0$, and there exist $\beta>b>0$ such that $f (s)>0 $  for $s>b$, $f (s)\leq 0$, $f (s) \not\equiv 0$ for $s\in \mbr{0,b}$, $F(\beta)=0$, and $\lim_{s\to \infty}F (s) =\infty $, where $ F(s)=\int_{0}^{s} f(t) ~\mathrm{d} t$;\vskip0.06in
		\item[($f_2$)] $f$ is continuous in $[0,\infty)$, continuously differentiable in $(0,\infty)$ and  $f'\in L^1(0,1) $;
		\vskip0.06in
		\item[($f_3$)]$\sbr{  \frac{F}{f}}'(s)\geq \frac{n-2}{2} $ for all $s\geq \beta$.\vskip0.06in
	\end{itemize}
	They proved that for  $2\leq n\leq 4$ and any $k\in \mathbb{N}$, there exists at most one radial solution  $u(r)=u(|x|)$  which has exactly $k$ sign changes in  $\sbr{0,\infty}$. These assumptions ($f_1$)--($f_3$) are satisfied by several concrete nonlinearities.
	For instance, the  piecewise linear nonlinearity \eqref{piecewise} for $n=3$; the combined nonlinearity $f(u)=-|u|^{q-1}u+|u|^{p-1}u$ ($0<q<p$), with no further restriction on $p,q$ for $n=2$, and $p^2+q^2\le 1$ for $n=3$; and the logarithmic nonlinearity $f(u)=u\log u^2$ for $n=2$.  
	
	\vskip 0.12in
	
	Our main result is the following. It provides a complete resolution of the Berestycki--Lions conjecture for the logarithmic nonlinearity, and also resolves Troy's open problem.

	\begin{theorem}\label{Thm}
		Let $n\geq 2$. Then for each positive integer $k\geq 1$, there exists a unique bound state $u=u(|x|)$ of
		\eqref{equlog} with
		precisely $k$ zeros for  $|x|>0$, up to translation and reflection.

	\end{theorem}

	Theorem \ref{Thm} is an immediate consequence of the following result, which provides a complete characterization of the solutions of \eqref{radialform}.    
	\begin{theorem}\label{theoremmain1}
		Let $n\geq 2$ and $u(r)$  be the solution of \eqref{radialform}. There exists a   sequence of initial data 
		\begin{equation}      \label{increasing prop}      e^{n/2}=\alpha_0<\alpha_1<\alpha_2<\cdots, \quad \text{ with }\quad  \lim_{k\to \infty} \alpha_k=\infty,
		\end{equation}
		such that,  for $u_k(r)$ denoting the solution with $u(0)=\alpha_k$,  the following statements hold.
		\begin{itemize}
			\medbreak
			\item[(i)] The Gausson $u_0(r) =\exp\sbr{-\frac{r^2}{2}+\frac{n}
				{2}} $ is the unique   ground state   of \eqref{radialform}.
			\medbreak
			\item[(ii)] For each $k\geq 1$, $u_k(r)$  is the unique bound state of \eqref{radialform} that has precisely $k$ (simple) zeros $0<z_1<z_2<\cdots<z_k<\infty$. Moreover,   $u_k(r)$ has a unique critical point $c_i$ in each $(z_i,z_{i+1})$ for $1\leq i \leq k-1$, and a unique critical point $c_k$ in $(z_k, \infty)$  beyond which $|u_k(r)|$ decreases strictly.   At each critical point, $|u_k|>e^{1/2}$. Furthermore, we have   the following asymptotic properties             
			\begin{equation} \label{decay1}
				\lim_{r\to\infty} \frac{  u_k'(r)}{u_k(r) \sqrt{|\log u_k^2(r)|}}=  \lim_{r\to\infty} \frac{  u_k'(r)}{ru_k(r)  }=-1 ,
			\end{equation}
			and  for any $ \epsilon \in \sbr{0,\frac{1}{2}}$,  
			\begin{equation} \label{decay2}
				\limsup_{r\to \infty}|u_k(r)|e^{(\frac{1}{2}-\epsilon)r^2}<\infty,\quad     \limsup_{r\to \infty}|u_k'(r)|r^{-1}e^{(\frac{1}{2}-\epsilon)r^2}<\infty.
			\end{equation} 
			\medbreak
			\item[(iii)]  If $\alpha = 1$, then $u(r) \equiv 1$. If $\alpha<\alpha_0$ and $\alpha \neq   1$, then $u(r) > 0$ in $\sbr{0,\infty}$ with $\inf u>0$ and $u(r)$ oscillates about $u \equiv 1$.
			\medbreak
			\item[(iv)]  If $\alpha \in \sbr{\alpha_k,\alpha_{k+1}}$ with $k\geq 0$, then   $u(r)$ is a nodal  solution
			with precisely $k + 1$ zeros, and oscillates about $u \equiv 1$ or $u \equiv -1$ behind its last zero.
		\end{itemize}
	\end{theorem}

	\vskip 0.1 in
	
	As mentioned earlier, statement (i) of Theorem \ref{theoremmain1} is   known for all $n\geq 2$; it 
	was proved by d'Avenia--Montefusco--Squassina \cite{avenia2014} for 
	$n\geq 3$ and by Troy \cite{TroyARMA2018} for $2\leq n\leq 9$ using different methods.  In this paper we provide a unified proof that works for every $n\ge 2$.

	The proof of Theorem \ref{theoremmain1} is based on the shooting method,   introduced by Kolodner \cite{Kolodner-CPAM} and later developed by many authors to study equations of the form\begin{equation}\label{generalradial}
		u'' +\frac{n-1}{r}u'+ f(u)=0, \quad   r>0 , \quad  u(0)=\alpha,  \quad u'(0)=0, 
	\end{equation} 
	with various nonlinearities; see,  e.g., \cite{Kwong1989,MS1987,TangInvent2026,Coffman1972,Coffman1996,Kwong-Zhang-DIE1991,TroyARMA2018} and the references therein.  The analysis centers on the function \begin{equation}
		v(r,\alpha):={\partial u(r,\alpha)}/{\partial \alpha},
	\end{equation}
	where $u(r)=u(r,\alpha)$ solves  \eqref{generalradial}.  
	The core  objective   is to show  that    $v(r)= v(r,\alpha)$ changes sign exactly once between any two consecutive zeros of   $u(r)$ when  $u(r)$ is a nodal solution. In the bound state case, $v(r)$ also changes sign once after the final zero of $u(r)$,  while for a ground state it changes sign exactly once on the whole interval $\sbr{0,\infty}$. In addition, it is essential to prove the divergence of $v$ at infinity, that is, 
	\begin{equation}\label{divergenceV}
		\lim_{r\to \infty} |v(r)|=\infty.
	\end{equation}
	In a recent pioneering work on the model nonlinearity $f(u)=-u+|u|^{p-1}u$,  Tang \cite{TangInvent2026}  developed a   framework  that achieves  this objective,  based on the behavior of auxiliary functions  $Q$, $M$, $T_1$, $T_2$, $B_a$ (see Section~\ref{SectionMainresult} for explicit expressions).  In the present paper, we   aim to extend  this framework   to investigate the uniqueness of bound states for the logarithmic nonlinearity  $f(u)=u\log u^2$ and prove Theorem \ref{theoremmain1}.  
	
	Implementing the roadmap proposed in \cite{TangInvent2026} leads to some additional technical difficulties.
	For instance, the proof of \eqref{divergenceV}  in that work  relies heavily on the following asymptotic behavior of ground states and bound states at infinity
	\begin{equation}
		\lim_{r\to \infty}  \frac{u'(r)}{u(r)}=-\sqrt{\zeta}, \quad \quad \limsup_{r\to \infty}|u(r)|e^{\sqrt{ \zeta-\epsilon}r}<\infty, \quad \text{ for any } \epsilon\in \sbr{0,\zeta},
	\end{equation}
	which    follows from condition  (C3)    $f'(0)=-\zeta<0$ for some finite $\zeta$; see \cite[Proposition 2.2 (iv)]{TangInvent2026} and also \cite[Lemma 5]{PS1983}. However, this approach fails in the present setting, primarily because the logarithmic nonlinearity  $f(u)=u\log u^2$  is not locally Lipschitz at the origin.  Indeed,   \begin{equation}
		f'(u)=   {\log u^2 +2 }\to-\infty,\quad  \text{ as }~~u\to0.
	\end{equation}
	Consequently, one cannot expect the same exponential decay rate as in the case $f'(0)>-\infty$. In fact, a formal phase‑plane analysis suggests that  the ratio $u'(r)/u(r)\to -\infty$ as   $r\to\infty$, and the argument in \cite{TangInvent2026} does not carry over directly.  New ideas and techniques are therefore required. 
	To overcome this difficulty, we directly investigate the asymptotic behaviors of ground states and bound states to \eqref{radialform}.  	As shown in Lemma \ref{lemma decay}, such solutions satisfy
	\begin{equation} 
		\lim_{r\to\infty} \frac{  u'(r)}{u(r) \sqrt{\abs{ \log u^2(r)}}}=-1 ,   \quad \lim_{r\to\infty} \frac{  u'(r)}{ru(r)  }=-1 ,
	\end{equation}
	and  for any $ \epsilon \in \sbr{0,\frac{1}{2}}$,  
	\begin{equation} \label{a22}
		\limsup_{r\to \infty}|u(r)|e^{(\frac{1}{2}-\epsilon)r^2}<\infty,\quad   \limsup_{r\to \infty}|u'(r)|r^{-1}e^{(\frac{1}{2}-\epsilon)r^2}<\infty.
	\end{equation}Moreover, the exponent $\frac{1}{2}$ in \eqref{a22} is  sharp; we refer to Remark \ref{sharpness} for details.
	We believe that  these asymptotic behaviors are of independent interest and may be useful in other contexts.

		\vskip0.12in 
		This paper is organized as follows. In Section \ref{Sectionpreliminary}, we present some preliminary results,
		including some known results  and the asymptotic behavior of ground states and bound states. In Section \ref{SectionMainresult}, we  prove Theorems \ref{Thm} and \ref{theoremmain1}.

		\section{Preliminaries}\label{Sectionpreliminary}

		In \cite{TangInvent2026}, the author established several useful properties for the equation $   u'' +\frac{n-1}{r}u'+ f(u)=0$ under general assumptions on $f$,  denoted by (C1)--(C9) in that paper.  The nonlinearity $f(u)=u\log u^2$ of interest here   satisfies all of these conditions except (C3),   i.e.,  $f'(0)=-\zeta<0$ for some finite $\zeta$. In this  section we recall some known results from \cite{TangInvent2026}, introduce notation,  and establish the asymptotic convergence and decay rates  at infinity of ground states and bound states. Throughout this section, we assume that $n\geq 2$ unless otherwise specified.
		
		\subsection{Some basic properties}
		Throughout the paper, we  take 	$f(u)=u\log u^2$ and  define the well-known energy functional 
		\begin{equation}\label{defi of E}
			E(r)= \frac{{u'}^2(r)}{2}+ F(u(r)), \qquad 	 F(u)=\int_{0}^{u} f(t) ~\mathrm{d} t=\frac{1}{2} u^2\sbr{\log  u ^2-1},  
		\end{equation}
		for which we have 
		\begin{equation}
			E'(r)= -\frac{n-1}{r}{u'}^2(r), \quad E'(0)=0, \quad \text{ and }\quad  E(0)=\frac{\alpha^2}{2} \sbr{\log \alpha^2-1} .
		\end{equation}
		In what follows, we collect    some basic properties of $E$ and $u$; see   Lemmas \ref{lemma propertyE}--\ref{lemma nodalcriticalpoint}.      The  proofs  can be found in   \cite[Propositions 2.1, 2.2, 2.3 and  2.4]{TangInvent2026}.  We mention that these proofs  do  not rely on the condition  (C3). 
		\begin{lemma}\label{lemma propertyE}
			If $\alpha=1$, then the unique solution of  \eqref{radialform} is $u\equiv 1$.  Now let $u(r)\not\equiv 1$ be a solution of \eqref{radialform}.  Then the following hold:
			\begin{itemize}
				\medbreak
				\item[(i)] $E(r)$ is strictly decreasing in $\sbr{0,\infty}$. Moreover, if $u$ is a ground state or a bound state, then $E(r)>0$ for all $r>0$.
				\medbreak
				\item[(ii)] If there exists $0<r_1<r_2$ such that $ |u(r_1)|=|u(r_2)|$, then $|u'(r_1)|>|u'(r_2)|$.
				\medbreak
				\item[(iii)] Define $\nm{u}_\infty:=\sup \lbr{ |u(r)|: r\geq 0}$,   and then  we have 
				\begin{equation}
					\nm{u}_\infty =\alpha, \quad \text{if }\alpha>1;\quad \quad \alpha< \nm{u}_\infty <\alpha_*, \quad \text{if }\alpha<1,
				\end{equation}
				where   $\alpha_*=e^{1/2}$ is the unique positive number satisfying  $F(\alpha_*)=0$. 
			\end{itemize}
		\end{lemma}
		\begin{lemma}\label{lemma oscalliate}
			Let   $u(r)\not \equiv 1$ be a solution of \eqref{radialform}.  If $E(\bar r)\leq 0$ and $u(\bar r)>0$ at some $\bar r\geq 0$, then $u(r)\in \sbr{0, \alpha_*}$ and it oscillates about $u\equiv1$ in $\sbr{\bar r, \infty}$: There is a sequence of critical points of $u$, labelled as $\tilde{c}_1<\tilde{c}_2<\cdots$, such that 
			\begin{equation}
				\alpha_*>u(\tilde{c}_1)>u(\tilde{c}_3)>\cdots>1 \quad \text{ and } \quad 0<u(\tilde{c}_2)<u(\tilde{c}_4)<\cdots<1.
			\end{equation}  If  $E(\bar r)\leq 0$ and $u(\bar r)<0$ at some $\bar r\geq 0$, then the same can be said for $-u$.
		\end{lemma}

		The following lemma  provides a classification of positive and nodal (sign-changing)  solutions of \eqref{radialform}. Note that the  ground state  corresponding to  a positive solution of \eqref{radialform} that satisfies  $u(r)\downarrow0$ as $r\to \infty$,  while a bound state is a nodal solution of \eqref{radialform} that satisfies $|u(r)|\downarrow 0$ as $r\to\infty$.
		
		\begin{lemma}\label{lemma positivenodalsolution}
			Let   $u(r)\not \equiv 1$ be a solution of \eqref{radialform}.  A positive solution $u $ is either a ground state, or an oscillatory function that oscillates about $1$ behind its last zero. A nodal solution $u$  is either a bound state, or an oscillatory function that oscillates about $1$ or $-1$ behind its last zero.
			
		\end{lemma}
		
		\begin{remark}
			{\rm We point out that in \cite[Proposition~2.4 (iv)]{TangInvent2026}, the asymptotic decay behavior of ground states and bound states is established using condition (C3). This behavior is crucial for proving uniqueness; however, in our setting condition (C3) is not satisfied. In the next subsection \ref{subsetionasymptotic}, we will therefore investigate the asymptotic convergence and decay rates at infinity of ground states and bound states. }
		\end{remark}

		\vskip0.04in
		For the ground state, we have the following result. 
		\begin{lemma} \label{lemma existcriticalpoint}
			If $u   $ is a ground state  of \eqref{radialform}, then  we have  $\alpha>\alpha_*=e^{1/2}$ and  $u'(r)<0$ in $\sbr{0,\infty}$.
		\end{lemma}
		\vskip0.04in

		Note that a nodal solution cannot have a double zero, since $u(r) = u'
		(r) = 0$ at any $r > 0$ would imply $u \equiv 0$ on $[0,\infty)$. From now on, we  define a $k$-node solution as a nodal solution of \eqref{radialform}  that has exactly $k$ zeros.   If a  $k$-node solution   of \eqref{radialform}   satisfies the asymptotic   condition \eqref{decayassum}, we call it a $k$-node bound state.
		\vskip0.08 in
		\begin{lemma}\label{lemma nodalcriticalpoint}
			If $u   $ is a nodal solution of \eqref{radialform}, then  $u$ has only finitely many sign changes. Suppose that $u$ has exactly $k\geq 1$ zeros $z_1<z_2<\cdots<z_k$. Then  $u$ has $k$ critical points in $[0,z_k]$, labelled as $0=c_0<c_1<\cdots<c_{k-1}$, with $c_i\in (z_i,z_{i+1})$, $1\leq i\leq k-1$, and $$\alpha=u(0)>|u(c_1)|>|u(c_2)|>\cdots>|u(c_{k-1})|>\alpha_*.$$ If in addition $u$
			is a bound state,  then there is a unique critical point $c_k>z_k$ with $$|u(c_{k-1})|>|u(c_{k})|>\alpha_*.$$
			
		\end{lemma}

		\subsection{The asymptotic convergence and decay rates  at infinity}\label{subsetionasymptotic}
		The asymptotic behavior of ground states and bound states described below plays a crucial role in proving Theorem \ref{theoremmain1}, and seems to be new in the study of logarithmic equations.
		\begin{lemma}\label{lemma decay}
			Let $u$ be a ground state or a bound state  of \eqref{radialform}. Then
			\begin{equation}\label{convergence}
				\lim_{r\to\infty} \frac{  u'(r)}{u(r) \sqrt{\abs{ \log u^2(r)}}}=-1 , \qquad  \lim_{r\to\infty} \frac{  u'(r)}{ru(r)  }=-1 .
			\end{equation}
			Moreover,  for any $ \epsilon \in \sbr{0,\frac{1}{2}}$, we have 
			\begin{equation}\label{decay}
				\limsup_{r\to \infty}|u(r)|e^{(\frac{1}{2}-\epsilon)r^2}<\infty,\qquad   \limsup_{r\to \infty}|u'(r)|r^{-1}e^{(\frac{1}{2}-\epsilon)r^2}<\infty.
			\end{equation}
			
		\end{lemma}
		\begin{proof}
			Define 
			$$\hat{u}(r):= -\frac{  u'(r)}{u(r) \sqrt{\abs{ \log u^2(r)}}}.$$
			Then   $\hat{u}(r)>0 $ for $r>c_k$, where we set $c_0=0$ when   $u$ is a ground state. Since $u$ tends to zero  as $r\to \infty$, we have $\log u^2(r)\to -\infty$.   Hence there exists  $\xi_1>0$ such that for all $r>\xi_1$,  
			\begin{equation}
				|u(r)|<1, \quad    \log u^2(r)<0, \quad \text{ and }~~\frac{1}{2}  \sqrt{\abs{ \log u^2(r)}}-\frac{1}{\sqrt{\abs{ \log u^2(r)}} } >2.
			\end{equation}
			It follows from \eqref{radialform} that
			\begin{equation}
				\begin{aligned}
					\hat{u}'&= \hat{u}^2 \sbr{  \sqrt{\abs{ \log u^2}}-\frac{1}{\sqrt{\abs{ \log u^2}} }}-\frac{n-1}{r}\hat{u}-\sqrt{\abs{ \log u^2}} \\
					&= \hat{u}^2 \sbr{ \frac{1}{2}  \sqrt{\abs{ \log u^2}}-\frac{1}{\sqrt{\abs{ \log u^2}} }}-\frac{n-1}{r}\hat{u}+\sbr{\frac{1}{2}\hat{u}^2 -1} \sqrt{\abs{ \log u^2}}\\
					&\geq 2 \hat{u}^2-\frac{n-1}{r}\hat{u}+\sbr{\frac{1}{2}\hat{u}^2 -1} \sqrt{\abs{ \log u^2}}, \quad \text{ for }~ r>\xi_1.
				\end{aligned} 
			\end{equation}
			If there exists some $\hat{r}>\max\lbr{2n,\xi_1,c_k}$ such that $\hat{u}>2$, then 
			\begin{equation}
				\hat{u}'\geq  \hat{u}^2+\sbr{   \hat{u}^2-\frac{1}{2}\hat{u}} +\sqrt{\abs{ \log u^2}}\geq \hat{u}^2   \quad \text{ at } ~~r=\hat{r},
			\end{equation}
			and this inequality continues to hold for all larger  $r>\hat{r}$ until $\hat{u}$ would become infinite at a finite value of $r$. This is impossible. Hence,  $\hat{u}\leq 2$ for all $r>\max\lbr{2n,\xi_1,c_k}$. Since $u$ and $u'$ tend  to zero  as $r\to \infty$, we know that $u^2  \abs{ \log u^2}\to0 $ as  $r\to \infty$. By using L'H\^opital's rule and \eqref{radialform}, we  have
			\begin{equation}
				\begin{aligned}
					\lim_{r\to\infty} 	\hat{u}^2&=	\lim_{r\to\infty} \frac{  u'^2}{u^2  \abs{ \log u^2}}=-\lim_{r\to\infty} \frac{  u'^2}{u^2   \log u^2} =-\lim_{r\to\infty} \frac{ u''}{u   \log u^2+u}\\&=\lim_{r\to\infty} \sbr{ -\frac{\sbr{n-1}\hat{u}}{r} \cdot\frac{\sqrt{\abs{ \log u^2}}}{  \log u^2+1} +\frac{  \log u^2}{   \log u^2+1}   }=1.
				\end{aligned}
			\end{equation}
			Therefore, $\hat{u} \to 1$ as $r\to \infty$, which implies that $\hat{u}>\sqrt{1-2\epsilon}$ for any $\epsilon\in \sbr{0,\frac{1}{2}}$ and sufficiently large $r$. Recall that   $|u(r)|<1$ for $r>\xi_1$. Then there holds 
			\begin{equation}\label{lpq}
				\hat{u}(r)= -\frac{  u'(r)}{u(r) \sqrt{\abs{ \log u^2(r)}}}=\frac{d}{dr} \sqrt{\abs{ \log u^2(r)}}.
			\end{equation}By integration we can obtain that \begin{equation} \label{s2}
				\limsup_{r\to \infty}|u(r)|e^{(\frac{1}{2}-\epsilon)r^2}<\infty.
			\end{equation}Finally, by using L'H\^opital's rule again, we have 
			\begin{equation}\label{s1}
				\begin{aligned}
					\lim_{r\to\infty} \frac{  u'(r)}{ru(r)  }&=	\lim_{r\to\infty}\sbr{\frac{  u'(r)}{u(r) \sqrt{\abs{ \log u^2(r)}}} \cdot  \frac{  \sqrt{\abs{ \log u^2(r)}}}{r}}\\&=- \lim_{r\to\infty} \frac{  \sqrt{\abs{ \log u^2(r)}}}{r}=-	\lim_{r\to\infty} \hat{u}(r)=-1.
				\end{aligned}	 
			\end{equation}Then the second relation in \eqref{decay} follows directly from \eqref{s2} and \eqref{s1}.
			This completes the proof.
		\end{proof}
		\begin{remark}\label{sharpness}
			{\rm The exponent $\frac{1}{2}$ in \eqref{decay} is sharp in the sense that for any $\epsilon>0$,
				\begin{equation}\label{decay22}
					\lim_{r\to\infty}|u(r)|e^{(\frac{1}{2}+\epsilon)r^2}=+\infty, \qquad   \lim_{r\to\infty}|u'(r)|r^{-1}e^{(\frac{1}{2}+\epsilon)r^2}=+\infty.
				\end{equation}
				Indeed, since  $\hat{u}(r) \to1$ as $r\to\infty$,    there exists  $R>\xi_1$ such that  $\hat{u}<\sqrt{1+\epsilon}$ for all $r>R$. Then integrating \eqref{lpq} from $R$ to $r$, there exists a constant $C$ (depending on $R$) such that   
				\[
				\sqrt{|\log u^2(r)|} < \sqrt{1+\epsilon}\,r + C.
				\] Squaring and using $\log u^2(r)<0$ yields
				\begin{equation}
					\log u^2(r)>-(1+\epsilon)r^2-2\sqrt{1+\epsilon}\,Cr-C^2.
				\end{equation}
				Consequently,
				\begin{equation}\label{q222}
					|u(r)| e^{(\frac{1}{2}+\epsilon)r^2} = \exp\Bigl(\frac{1}{2}\log u^2(r) + \bigl(\tfrac12+\epsilon\bigr)r^2\Bigr) > \exp\Bigl(\frac{\epsilon}{2}r^2 - \sqrt{1+\epsilon}\,C r - \frac{C^2}{2}\Bigr) \to +\infty,
				\end{equation}
				as  $r\to\infty$.  The second  limit in \eqref{decay22} follows directly from the first one together with \eqref{s1}. }
		\end{remark}

		\subsection{Pohozaev function and its variants}\label{subsectionPohozaev}
		Consider the Pohozaev function \cite{pohozaev1965} associated with \eqref{radialform}
		\begin{equation} \label{defi of P}
			\begin{aligned}
				P(r)&=2r^n	E(r)+(n-2)r^{n-1}uu'=r^n	[ u'^2+2F(u)]+(n-2)r^{n-1}uu',
			\end{aligned}
		\end{equation}
		and the following related functions
		\begin{equation} \label{defi of P1}
			P_1(r)=r^n	[ u'^2+uf(u)]+(n-2)r^{n-1}uu',
		\end{equation}
		and 
		\begin{equation} \label{defi of P2}
			P_2(r)=r^n	[ u'^2+\frac{n-2}{n}uf(u)]+(n-2)r^{n-1}uu'.
		\end{equation}
		Then we have the following.

		\begin{lemma} \label{lemma positivefunction}
			Let   $u$ be  the solution of \eqref{radialform}. Then  the following statements hold  for   $r\in (0,z_k]$ if $u$ is a $k$-node solution, and for $r\in (0,\infty)$ if $u$ is a ground state or a bound state.   For $n\geq 2$, we have   
			\begin{equation}
				P(r)>0, \quad P_1(r)>0, 
			\end{equation}  
			and $\omega'(r)>0$ whenever  $u(r)\neq 0$, where    $\omega (r)=-{ru'(r)}/{u(r)}$. Moreover, for   $n\geq 3$  we also have
			\begin{equation}
				P_2(r)>0, \quad  -\mbr{{P(r)}/{r^n}}' >0.  
			\end{equation}
			
		\end{lemma}
		\begin{proof}
			This lemma follows the same strategy of \cite[Proposition 2.5]{TangInvent2026}, which deals with the general nonlinearity case.  Although the logarithmic nonlinearity $u\log u^2$ does not satisfy condition (C3) in \cite{TangInvent2026}, the proof can nevertheless be carried out with the aid of Lemma \ref{lemma decay}. For completeness, we sketch it here.  Observe that if $u$ is a $k$-node solution with zeros $z_1<\cdots<z_k$, then $	P(z_i)= P_1(z_i)=	P_2(z_i)= z_i^n u'^2(z_i)>0$ for $1\leq i\leq k$. Moreover, by Lemma \ref{lemma decay}, we know that $|u(r)|$ and  $|u'(r)|$ decay to 0 with a Gaussian-type decay behavior.  Hence,  
			\begin{equation} \label{decayeee}
				\lim_{r\to\infty} r^nu'^2=0, \quad 	\lim_{r\to\infty}r^{n-1}uu'=0, \quad \lim_{r\to\infty} r^nu^2=0, \quad \lim_{r\to\infty} r^nu^2\log u^2=0,
			\end{equation}
			where the last relation follows from the  inequality  $ |s^2\log s^2|\leq |s|$ for $|s|$ small enough. Consequently, we know that
			$P(r)$, $P_1(r)$, and  $P_2(r)$  all approach zero as $r\to \infty$.   Then the proof follows word for word that of  \cite[Proposition 2.5]{TangInvent2026}.  
		\end{proof}
		
		
		\section{Proof of the main results} \label{SectionMainresult}
		
		This section is dedicated to the proof of Theorems \ref{Thm} and \ref{theoremmain1}. 
		Note that the  logarithmic   nonlinearity $f(u)=u\log u^2$  is continuous on $[0,\infty) $, continuously differentiable on $(0,\infty)$ and satisfies   $f'\in L^1(0,1) $. 
		Then $u(r,\alpha)$ and $u'(r,\alpha)=\frac{\partial}{\partial r}u(r,\alpha)$ are of class  $C^1$ in $\sbr{0,\infty}\times \sbr{1,\infty}$; see \cite{Cortazar2009}. We set 
		\begin{equation}
			v(r,\alpha):=\partial u(r,\alpha)/\partial \alpha.
		\end{equation}
		Then   for any $r>0$ such that $u(r)\neq 0$, $v$ satisfies the linear differential equation    
		\begin{equation}\label{equv}
			v''+\frac{n-1}{r} v' + (\log u^2+2) v=0,  \quad v(0)=1, \quad v'(0)=0.
		\end{equation}
		In order to characterize the sign-changing structure of   $v$ and its  asymptotic behavior   at infinity, we  introduce some  important auxiliary functions:
		\begin{equation} \label{defi of T2}
			\begin{aligned} 
				&Q(r)=r^{n}\mbr{u'v'+f(u)v}+(n-2)r^{n-1}u'v, \\
				& M(r)=r^{n-1}\sbr{u'v-uv'}, \\
				&	T_1(r)=Q(r)-g_1(u)M(r),  \\
				&  T_2(r)=Q(r)-g_2(u)M(r), 
			\end{aligned} 
		\end{equation}
		where 
		\begin{equation}
			\begin{aligned}
				&g_1(u)=2f(u)/\mbr{uf'(u)-f(u)}=\log u^2,\\
				& g_2(u)=2F(u)/\mbr{uf(u)-2F(u)}=\log u^2-1.
			\end{aligned}
		\end{equation}
		Obviously, 
		\begin{equation}
			g_2(u)=g_1(u)-1, \quad \text{ and } \quad T_2(r)=T_1(r)+M(r).
		\end{equation}
		The auxiliary functions $Q$, $M$, $T_1$, $T_2$  have been widely used for power-type nonlinearities (see e.g. \cite{TangInvent2026, Coffman1972, FrankSummer, Kwong1989, M1993, Tao2006, TangJDE2003}). For the logarithmic nonlinearity $f(u)=u\log u^2$
		considered here, however, the sign of these functions does not follow directly from those results. A careful re-evaluation is required, because the logarithmic dependence alters the sign relations that are crucial for the analysis of $v$. Nevertheless, the same functional framework remains essential, and we therefore revisit the sign conditions in the logarithmic setting.
		Most of this section is devoted to determining the signs of $Q$, $M$, $T_1$, $T_2$. To this end,  we define  
		\begin{equation}\label{defi of Qi}
			Q_i(r)=Q(r)+ir^{n-1}u'v,
		\end{equation}   and use the following derivative formulas, 
		which follow from \cite[Appendix~A.3]{TangInvent2026}:
		\begin{equation}\label{defi of deriQMT}
			\begin{aligned}
				&Q'(r)=2r^{n-1} f(u)v=2r^{n-1} \sbr{u\log u^2 }v,\\&Q_i'(r)=r^{n-1}[iu'v'+(2-i)f(u)v], \\
				&M'(r)=r^{n-1}\mbr{uf'(u)-f(u)}v=2r^{n-1}uv,\\
				&T_1'(r)=-g_1'(u)u'M(r)=-\frac{2u'}{u} M(r),\\
				& T_2'(r)=T_1'(r)+M'(r)=-\frac{2u'}{u}   M(r)+2r^{n-1}uv= \frac{2}{r} \mbr{ \omega(r)M(r)+r^{n}uv },
			\end{aligned}
		\end{equation}
		where $\omega(r)$ is defined in Subsection \ref{subsectionPohozaev}.

		\subsection{The analysis of the ground state}
		In the proof
		of the uniqueness of ground states in $\R^n$, $n\geq 2$, it is 
		essential  to show that $v$ changes sign  exactly  once on $\sbr{0,\infty}$, and that $v$  diverges at infinity. By Lemma \ref{lemma existcriticalpoint}, there exists a unique $r_1\in \sbr{0,\infty}$ such that $u(r_1)=1$. Then we have the following.
		
		\begin{proposition}\label{prop vgroundstate}
			Let $n \geq 2$ and  $u=u(r,\alpha)$ be a ground state  of \eqref{equlog}.  Then there exists $\tau\in \sbr{0,r_1}$   such that  
			\begin{equation}\label{j000}
				v(r)>0 ~ \text{ in } \sbr{0,\tau}, \quad 	v(\tau)=0, \quad 	v(r)<0~ \text{ in } \sbr{ \tau,\infty}.  
			\end{equation}
			Moreover,   $v(r)$  is strictly decreasing for sufficiently large $r$ and 
			\begin{equation}
				\lim_{r\to\infty} 	v(r)=-\infty.
			\end{equation}
			Furthermore, we have 
			\begin{equation}\label{i2}
				Q(r),~ M(r),~  T_1(r),~   T_1'(r)>0 ~~ \text{ for } ~r \in \sbr{0,\infty}.
			\end{equation}
		\end{proposition}
		\begin{proof}
			Consider  the Wronskian of $u'$ and $v'$ 
			\begin{equation}\label{defi of W}
				W(r):=r^{n-1}(u''v'-u'v'')=r^{n-1}\mbr{f'(u)u'v-f(u)v'}.
			\end{equation}
			Then  by the Wronskian identity  \cite[(B.16)]{Tao2006}, we have 
			\begin{equation}\label{deri w}
				W'(r)=r^{n-1}f''(u)u'^2v=2r^{n-1}(u'^2/u)v.
			\end{equation}We first prove  that $v(r)$  changes sign in $(0,r_1)$. If the statement is false,  then $v>0$ in $(0,r_1)$. It follows from \eqref{deri w}  and $u>1$ in $(0,r_1)$ that   $$W'(r)=  \frac{2r^{n-1}u'^{2}v}{u} >0 \quad \text{in} ~~(0,r_1).$$   Hence, we have  $W(r_1)>W(0)=0$. However, since  $f(u(r_1))=f(1)=0$, $f'(u(r_1))=f'(1)>0$, $u'(r_1)<0$, and $v(r_1)>0$, we have $W(r_1)\leq0$. This is a contradiction.  Therefore,  $v(r)$   changes sign in $(0,r_1)$, and  there is $\tau\in(0,r_1)$ such that 
			\begin{equation} \label{j11}
				v(r)>0 ~ \text{ in } \sbr{0,\tau}, \quad 	v(\tau)=0, \quad 	v(r)<0~ \text{ in } \sbr{ \tau,r_1}, \quad \text{ and }~~ v'(\tau)<0. 
			\end{equation} 
			Then  it follows from \eqref{defi of deriQMT} that $ 	Q'(r)>0$ and $	M'(r)>0$ for $ r\in \sbr{0,\tau } $. By Lemma \ref{lemma existcriticalpoint}, we have $u'<0$ in $\sbr{0,\infty}$.
			Since $Q=M=T_1=0$ at $r=0$, we obtain  
			\begin{equation}\label{ik3}
				Q(r), ~ M(r), ~ T_1(r),~  T_1'(r)>0, \quad \text{ for } ~~r\in (0,\tau ].
			\end{equation}
			We now prove that 
			\begin{equation}\label{i1}
				M(r)>0, \quad \text{ for } ~r \in \sbr{0,\infty}. 
			\end{equation}
			With this   result, we can show that $\tau$ is the unique zero of  $v(r)$ on $\sbr{0,\infty}$; thus the first part of \eqref{j000} holds.    Indeed, suppose $v$ has a second zero $\xi$ in $(\tau,\infty)$. Then  $v(\xi)=0$ and $v'(\xi)>0$, which implies $M(\xi)=-\xi^{n-1}u(\xi)v'(\xi)<0$. This contradicts \eqref{i1}.

			\vskip0.04in

			If the statement \eqref{i1} is false,  then by \eqref{ik3} there exists $\bar{\tau}\in \sbr{\tau,\infty }$ such that 
			\begin{equation} \label{s3}
				M(r)>0 ~\text{ in } \sbr{0,\bar{\tau}},\quad 	M(\bar{\tau})=0, ~~ \text{ and } ~~v(r)<0  ~\text{ in }\sbr{\tau,\bar{\tau}}.
			\end{equation}
			Therefore,  $T_1'(r)>0$ for all $r\in \sbr{0,\bar{\tau}}$,   from which we obtain
			\begin{equation}
				Q(\bar{\tau})=	T_1(\bar{\tau})+g_1(u(\bar{\tau}))M(\bar{\tau})=	T_1(\bar{\tau})>T_1(0)=0.
			\end{equation}
			Hence, $\bar{\tau}$ cannot be a zero of $v$ and $v(\bar{\tau})<0$. From  $	M(\bar{\tau})=0$, we know that  $u'v=uv'>0$ and $u'/v'=u/v<0$ at $r=\bar{\tau}$. Hence,  
			\begin{equation}
				\begin{aligned}
					P_1(\bar{\tau})&=\bar{\tau}^n	[ u'v' \cdot\frac{u'}{v'}+ f(u)v\cdot\frac{u}{v}]+(n-2)\bar{\tau}^{n-1}uv' \frac{u'}{v'}\\
					&=Q(\bar{\tau})\cdot\frac{u(\bar{\tau})}{v(\bar{\tau})}<0,
				\end{aligned}
			\end{equation}
			which contradicts Lemma \ref{lemma positivefunction} and proves \eqref{i1}. Hence,  $\tau$ is the unique zero of  $v(r)$ on $\sbr{0,\infty}$ such that  
			\begin{equation}
				v(r)>0 ~ \text{ in } \sbr{0,\tau}, \quad 	v(\tau)=0, \quad 	v(r)<0~ \text{ in } \sbr{ \tau,\infty}.
			\end{equation}
			Moreover,  we deduce from  \eqref{defi of deriQMT}  and \eqref{i1} that 
			\begin{equation}\label{i3}
				T_1'(r)>0~~ \text{ and }~~ T_1(r)>T_1(0)=0, ~~\text{ for all }~r>0.
			\end{equation}
			From which, together with the facts $f(u(r_1))=f(1)=0$ and  $g_1(u(r_1))=0$, we have $Q(r_1)=T(r_1)>0$.   By \eqref{defi of deriQMT}, we know that $Q(r)$   increases from $Q(0)=0$ in $(0,\tau)$, decreases  in $(\tau,r_1)$, and then increases in $(r_1,\infty)$. Then we obtain that
			\begin{equation}\label{i4}
				Q(r)>0, ~~\text{ for all }~r>0.
			\end{equation}
			Finally, we show that $$\lim_{r\to\infty} 	v(r)=-\infty.$$  Since $u$ is positive for all $r>0$ and  approaches  $0$ as $r\to \infty$, it follows from \eqref{equv} that, for   $r   $ large enough
			\begin{equation} \label{o999}
				\sbr{r^{n-1}v'}'=-r^{n-1}(\log u^2+2) v<0,
			\end{equation}
			so $r^{n-1}v'$ is strictly decreasing for all large  $r   $.  Consequently, 
			$v$  is monotone for  sufficiently large $r   $. Suppose, for contradiction, that $v(r)$ is bounded. Then $v(r)$  admits a finite limit $  L:=\lim_{r\to\infty} v(r)$.  Now we claim that 
			\begin{equation}\label{lim_of_Q}
				\lim_{r\to\infty}Q(r)=0.
			\end{equation}
			Indeed, from the decay estimate \eqref{decay}, we have  
			\begin{equation}
				\lim_{r\to\infty}  r^n f(u)v =     \lim_{r\to\infty}  r^n (u \log u^2) v=0, \qquad \lim_{r\to\infty} r^n u'v=0.
			\end{equation}
			It remains to show that 
			\begin{equation}
				\lim_{r\to\infty}r^nu'v'=0.
			\end{equation}
			By \eqref{o999}, we know that $r^{n-1}v'$ is strictly monotonically decreasing for  sufficiently large $r   $. We distinguish two cases. 
			\vskip0.1in
			\noindent {Case 1}. $r^{n-1}v'$ has a finite limit. Then, using    the decay estimate \eqref{decay}, we obtain
			\begin{equation}     
				\lim_{r\to\infty}r^nu'v'=\lim_{r\to\infty}ru'(r^{n-1}v')=0.
			\end{equation} 
			{Case 2}. $r^{n-1}v'$ is unbounded and tends to $-\infty$.  By L'H\^opital's rule, we deduce from \eqref{convergence}, \eqref{o999} that $$\lim\limits_{r\to\infty}\frac{v'}{r^3}=\lim\limits_{r\to\infty}\frac{r^{n-1}v'}{r^{n+2}}=-\frac{L}{n+2}\lim\limits_{r\to\infty}\frac{\log u^2+2}{r^2}=-\frac{L}{n+2}\lim\limits_{r\to\infty}\frac{u'(r)}{ru(r)}=\frac{L}{n+2},$$ 
			Then by \eqref{decay}, we obtain that 
			\begin{equation}  \lim_{r\to\infty}r^nu'v'=\lim_{r\to\infty}r^{n+3}u'\cdot \frac{v'}{r^3}=0.
			\end{equation}
			Thus the claim \eqref{lim_of_Q} is proved.  
			However, by \eqref{i4} and  $Q'(r)>0$ for $r>r_1$,  we know that $\lim_{r\to\infty}Q(r)> Q(r_1)>0$, which contradicts \eqref{lim_of_Q}.  Hence, we have  $\lim_{r\to\infty} 	v(r)=-\infty$.  This completes the proof. 
			
			%
			%
		\end{proof}

		\subsection{The analysis of  nodal solutions}
		
		We now turn to the analysis of nodal solutions. Throughout this subsection, we assume that $n\geq 3$. Unlike the ground state case, a nodal solution has multiple zeros and changes sign, which makes the analysis more delicate. Let $u$ be a nodal solution of \eqref{radialform} with exactly  $k$ zeros $z_1<\cdots<z_k$, $k\geq 1$. Then by Lemma \ref{lemma positivenodalsolution}, $u$ is either a bound state or an oscillatory nodal solution. By  Lemma  \ref{lemma nodalcriticalpoint}, we can decompose  $\sbr{0,\infty}$ into  the following intervals
		\begin{equation}
			(0,\infty)=\begin{cases}
				\bigcup_{i=1}^{k-1} {(c_{i-1},c_i]}\cup {(c_{k-1},z_k)}\cup {(z_k,\infty)}, &\quad \text{ if } ~ u ~\text{ oscillates   in  }\sbr{z_k,\infty},\\
				\bigcup_{i=1}^k  {(c_{i-1},c_i]}  \cup  {(c_k,\infty)}, &\quad \text{ if } ~ u~ \text{  is a bound state},
			\end{cases}
		\end{equation}
		where $c_0=0$. In this subsection, we  borrow some ideas from \cite{TangInvent2026} to investigate the behavior of $v$ on these intervals.  A crucial role is played by the bridge function  $B_a(r)$ 
		\begin{equation}\label{defi of Ba}
			B_a(r)=Q(r)-aM(r)-2F_a(u(r)) \cdot \frac{r^{n-1}v}{u'}, \quad \text{ and } \quad B_a'(r)=-2F_a(u(r))\cdot \frac{Q_n(r)}{ru'^2},
		\end{equation}
		where $a$ is an arbitrary constant, $Q_n(r)$ is defined in \eqref{defi of Qi} and 
		\begin{equation}
			F_a(u(r)) =F(u(r))-\frac{a}{2}\mbr{uf(u)-2F(u)}.
		\end{equation}
		Within a single phase $(c_{i-1},c_i]$, the functions $u$, $f(u)$, and $F(u)$ all change signs. To handle this, we introduce points satisfying
		\begin{equation}
			c_{i-1}<b_i<r_i<z_i<\bar{r}_i<\bar{b}_i<c_i,
		\end{equation}
		and define the absolute values at these points as
		\begin{equation}
			|u(b_i)|=|u(\bar{b}_i)|=\alpha_*, \quad |u(r_i)|=|u(\bar{r}_i)|=1,
		\end{equation}
		where   $\alpha_*=e^{1/2}$ is the unique positive number satisfying  $F(\alpha_*)=0$. 
		
		\vskip0.08in
		
	First, we show that  $v$  changes sign  exactly once on  $(0,z_1]$.

	\begin{proposition}\label{prop phase1 one zero}
		Let $u=u(r,\alpha)$ be    a  nodal solution of \eqref{radialform}. Then there exists $\tau_1\in(0,r_1)$ such that $v(r)>0$ in $(0,\tau_1)$, $v(\tau_1)=0$, and $v(r)<0$ in $(\tau_1,z_1]$. Moreover, we have 
		\begin{equation}
			Q(r),~M(r)>0,~~r\in(0,z_1];\quad T_1(r),~T_1'(r)>0,~~r\in(0,z_1).
		\end{equation}
	\end{proposition}\vskip0.08in
	The proof of Proposition~\ref{prop phase1 one zero} is similar to that of Proposition~\ref{prop vgroundstate}, the difference being that we replace $\infty$ with $z_1$; we omit the details.  Proposition~\ref{prop phase1 one zero} shows that if  $u$ is an oscillatory nodal solution with exactly one zero, then  $v$ changes sign exactly once on $(0,z_1]$.
	
	\vskip0.08in
	If $u$ is either an oscillatory nodal solution   with  at least two zeros or a bound state, then $u$ admits a critical point $c_1>z_1$,  and we   prove that $v$ changes sign exactly once on $(0,c_1]$. 
	
	\begin{proposition}\label{prop phase1}
		Let $u=u(r,\alpha)$ be  an oscillatory nodal solution of \eqref{equlog} with  multiple zeros or a bound state. Then there exists $\tau_1\in \sbr{0,r_1}$   such that  \begin{equation}\label{i6}
			v(r)>0 ~ \text{ in } \sbr{0,\tau_1}, \quad 	v(\tau_1)=0, \quad 	v(r)<0~ \text{ in } \sbr{ \tau_1, c_1}.
		\end{equation}Moreover, we have 
		\begin{equation}\label{i5}
			Q(r)>0~ \text{ in } (0,z_1]\cup \mbr{\bar{b}_1,c_1},\quad ~     M(r),~  Q_1(r),~ Q_2(r) >0 ~~ \text{ in }  (0,c_1],
		\end{equation}
		and
		\begin{equation}\label{u11}
			T_1(r),~   T_1'(r)>0 ~~\text{ in } ~ ( 0,z_1), 
		\end{equation}
		as well as
		\begin{equation} \label{p0}
			uf(u),~ f'(u),~ uv,~ u'v'>0, \quad \text{and } \quad uu',~uv',~u'v,~vv'<0 \quad \text{ in } \sbr{0,\tau_1}.
		\end{equation}
	\end{proposition}
	\begin{proof}
		By  Proposition \ref{prop phase1 one zero}, there exists $\tau_1 \in \sbr{0, r_1}$ such that   $	v(r)>0 $  in $ \sbr{0,\tau_1}$, $v(\tau_1)=0$, and $v(r)<0$ in $( \tau_1,z_1]$. Moreover, we have
		\begin{equation}\label{i7} 
			Q(r),~ M(r)>0 ~~ \text{ in } ~ ( 0,z_1]  \quad \text{ and  }\quad T_1(r),~   T_1'(r)>0~~\text{ in } ~ ( 0,z_1).
		\end{equation}
		Then \eqref{u11} holds. To prove \eqref{i6}, our first step is to demonstrate that
		\begin{equation}
			v(r)<0~ \text{ in } \mbr{ z_1, \bar{b}_1}.
		\end{equation}
		For this it is enough to show
		\begin{equation}\label{q1}
			Q_1(r)=Q(r)+r^{n-1}u'v>0\quad \text{ in } [z_1, \bar{b}_1],
		\end{equation}
		because at the first possible zero $z_p$ of $v$ in $(z_1,\bar{b}_1]$ we have $Q_1(z_p)=z_p^{n}u'(z_p)v'(z_p)<0$. Observe that since $u'(z_1)v(z_1)>0$, we have $Q_1(z_1)>Q(z_1)>0$.   If statement \eqref{q1} does not hold, then there exists some $\tilde{z}\in (z_1, \bar{b}_1]$ such that 
		\begin{equation}\label{q2}
			Q_1(r)>0 ~~ \text{ in } (z_1,  \tilde{z}), \quad  Q_1(\tilde{z})=0, \quad \text{ and }\quad  v<0  ~~\text{ in } \mbr{z_1,  \tilde{z}}.
		\end{equation}
		Given that  $u'v>0$ and $F(u)<0$ in $\sbr{z_1,\bar{b}_1}$, it follows that $Q_n(r)>Q_1(r)>0$ and $ B_0'(r)=-2F(u(r))\cdot  {Q_n(r)}/\sbr{ru'^2}>0$. Hence, we have 
		\begin{equation}
			B_0(\tilde{z})> B_0(z_1)=Q(z_1)>0.
		\end{equation}
		However, from \eqref{q2} we have 
		$Q(\tilde{z})=-\tilde{z}^{n-1}u'(\tilde{z})v(\tilde{z})$. Using \eqref{defi of Ba} together with Lemma \ref{lemma propertyE}, we obtain
		\begin{equation} \label{oi3}
			\begin{aligned}
				B_0(\tilde{z})&=-\tilde{z}^{n-1}u'(\tilde{z})v(\tilde{z})-2F(u(\tilde{z})) \cdot \frac{\tilde{z}^{n-1}v(\tilde{z})}{u'(\tilde{z})}=-\frac{2\tilde{z}^{n-1}v(\tilde{z})}{u'(\tilde{z})}\cdot E(\tilde{z})<0,
			\end{aligned}
		\end{equation}
		which is a contradiction,  and therefore \eqref{q1} is proved. Hence,  on $ \mbr{ z_1, \bar{b}_1}$, we have $	v(r)<0 $.     Note that $Q_n(r)>Q_1(r)>0$ on $ \mbr{ z_1, \bar{b}_1}$, while 
		$ B_0'(r)>0$ on $(z_1,\bar{b}_1)$. Consequently,
		\begin{equation}\label{q3}
			Q(\bar{b}_1)=B_0(\bar{b}_1)>B_0(z_1)= Q(z_1)>0.
		\end{equation}
		With \eqref{q3} at our disposal, we can now prove 
		\begin{equation}\label{q4}
			v(r)<0~ \text{ in } \mbr{   \bar{b}_1,c_1}.
		\end{equation}
		Suppose, to the contrary, that $v$ has a first possible zero at some $\tilde{l}$ in $  ( \bar{b}_1,c_1]$ . Then $v(r)<0$ in  $ ( \bar{b}_1,\tilde{l})$, $v(\tilde{l})=0$, and $v'(\tilde{l})>0$. Because    $u(r)<-\alpha_*<-1$ for all $r\in \mbr{   \bar{b}_1,c_1}$, we obtain 
		$Q'(r)=2r^{n-1} (u\log u^2) v>0 $ for $r\in ( \bar{b}_1,\tilde{l})$, so that 
		$Q(\tilde{l})>Q(\bar{b}_1)>0$. On the other hand,  by the definition of $Q$, we find    $Q(\tilde{l})=\tilde{l}^n u'(\tilde{l})v'(\tilde{l})<0$, which is a contradiction. Thus \eqref{q4} holds, which confirms \eqref{i6}.
		\medbreak
		Finally, we prove \eqref{i5}   and \eqref{p0}. As noted in  \eqref{i7},  both  $Q $ and $ M$ are positive on  $ (0,z_1]$. Since $v(r)<0$ and $u(r)<-\alpha_*<-1$  on $\mbr{   \bar{b}_1,c_1}$,  it follows   that $Q'(r)=r^{n-1} (u\log u^2) v>0$ on this interval, and then from   \eqref{q3} we obtain  $Q(r)>Q(\bar{b}_1)>0$ on $\mbr{   \bar{b}_1,c_1}$. Therefore, 
		\begin{equation}\label{q8}
			Q(r)>0~ \text{ in } (0,z_1]\cup \mbr{\bar{b}_1,c_1}.
		\end{equation} On $\mbr{z_1,c_1}$,  we know that $  M'(r)>0$, and therefore,    $M(r)>M(z_1)>0$. It remains to show that $ Q_1,Q_2 >0$ on $(0,c_1]$. By \eqref{defi of deriQMT}, we know that 
		\begin{equation}
			Q_1'(r)=r^{n-1}[u'v'+f(u)v], ~~ \text{and }~~Q_2'(r)=2r^{n-1}u'v'.
		\end{equation}  On  $(0,\tau_1)$, we have 
		\begin{equation}\label{p2}
			u>1,~ \text{ and } ~f(u),~ f'(u),~v>0.
		\end{equation}  
		Then by \eqref{equv},  we obtain 
		\begin{equation}\label{equv2}
			(r^{n-1}v')'=-r^{n-1}f'(u)v<0,
		\end{equation}
		which implies that $v'<0$, and hence, $Q_1',Q_2'>0$. Consequently, $Q_i(r)>Q_i(0)=0$ for $i=1,2$ on $(0,\tau_1)$. For $r \in [\tau_1,c_1)$, we have      $u'v>0$.    Then from \eqref{q1} we deduce that  $Q_2(r)> Q_1(r)>0 $ for $r \in \mbr{z_1, \bar{b}_1}$ and from \eqref{q8}  that $Q_2(r)>Q_1(r)>Q(r)>0$ for $r\in [\tau_1,z_1]\cup [ \bar{b}_1,c_1) $. Moreover, it is clear that $Q_2(c_1)=Q_1(c_1)=Q(c_1)>0$.  This completes the proof of \eqref{i5}. Finally,  \eqref{p0} follows directly from \eqref{p2} and  the fact that $u'<0$, $v'<0$ on  $(0,\tau_1)$. 
	\end{proof}
	\vskip0.1in
	
	In what follows, we investigate the behavior of $v$ on the remaining intervals.  The key tool is the function $T_2(r)$ defined in \eqref{defi of T2}.

	\begin{proposition}\label{prop phasei}
		Let $u=u(r,\alpha)$ be  a nodal solution of \eqref{equlog}   with $k\geq 2$ zeros.  Then for any $i\in \lbr{2,\ldots,k}$, and also for $i=k+1$ if $u$ is a bound state, we have  $$ Q(c_{i-1}),~ M(c_{i-1}),~ T_2(c_{i-1})>0.$$ Moreover, the variation $v$ has a unique zero $\tau_i$  on the  interval $I_i$ defined by
		\begin{equation}
			I_{i}=    \begin{cases}
				\mbr{c_{i-1},c_i},  \quad &\text{ if }  ~~i\in \lbr{2,\ldots,k-1} ,\\
				[c_{k-1},c_k],\quad &\text{ if }~~ i=k~ \text{ and } ~u ~\text{ is   a bound state},\\
				\mbr{c_{k-1},z_k},  \quad &\text{ if }~~ i=k ~ \text{ and } ~ u ~\text{ is not a bound state}, \\ [c_{k},\infty),\quad &\text{ if }~~ i=k+1~ \text{ and } ~u ~\text{ is   a bound state},
			\end{cases}
		\end{equation}
		with  $\tau_i\in \sbr{c_{i-1},r_i}$ in all cases.
	\end{proposition} 
	
	The rest of this subsection is devoted to the proof of Proposition \ref{prop phasei} by mathematical induction.  To carry out the induction, we first prove a lemma that describes the sign properties of 
	$v$ on a generic interval.
	
	\begin{lemma}\label{lemma sign}
		Assume that 
		\begin{equation}\label{assumption1}
			Q(c_{i-1}), ~ M(c_{i-1}) >0  ~\text{ for }~i\in \lbr{2,\ldots,k}, \text{  and also for }  ~ i=k+1 \text{ if }~u ~\text{ is   a bound state}.
		\end{equation}
		Then $v(r)$ changes sign  in $(c_{i-1},r_i)$ for  each such $i$.  Let $\tau_i$ be the first zero of $v$ in $(c_{i-1},r_i)$.    Then 
		\begin{equation} \label{p1}
			uf(u),~ f'(u),~ uv,~ u'v'>0, \quad \text{and } \quad uu',~uv',~u'v,~vv'<0\quad \text{ in } \sbr{c_{i-1},\tau_i}.
		\end{equation} 
		Moreover,  if we  further assume that 
		\begin{equation}\label{assumption2}
			Q_1(r)>0, \quad  \text{ for } ~~r \in  
			I_{i} ,
		\end{equation}
		then $\tau_i$ is the unique zero of $v$ on the interval $I_i$.
	\end{lemma}
	
	\begin{proof}
		We first prove that $v(r)$ changes sign in $(c_{i-1}, r_i)$ for $i\in \lbr{2,\ldots,k}$, and also for $i=k+1$ if $u$ is a bound state. Without loss of generality, we assume by contradiction that $v>0$ in $(c_{i-1}, r_i)$; the case $v<0$ can be treated similarly. Then by \eqref{deri w}, we have $W'(r)>0$ in $(c_{i-1}, r_i)$. Recall that $u'(c_{i-1})=0$ and  $|u(r_i)|=1$, and we have 
		\begin{equation}
			W(c_{i-1})=-c_{i-1}^{n-1}f(u(c_{i-1}))v'(c_{i-1}),  \quad \text{and}\quad  W(r_i)=r_i^{n-1}f'(1)u'(r_i)v(r_i).
		\end{equation}
		Moreover, it follows from  $Q(c_{i-1})$, $  M(c_{i-1}) >0$ that 
		\begin{equation} 
			f(u)v>0, \quad \text{ and } \quad uv'<0, \quad \text{ at } ~r=c_{i-1}.
		\end{equation}
		Together with $|u(c_{i-1})|>\alpha_*>1$, this implies that 
		\begin{equation}\label{q9}
			v(c_{i-1})v'(c_{i-1})<0, \quad \text{ and }\quad u(c_{i-1})v(c_{i-1})>0.
		\end{equation}
		Because $v > 0$ in $(c_{i-1}, r_i)$, we have $v(c_{i-1}) > 0$ and consequently $v'(c_{i-1}) < 0$. If $i$ is even, then $f(u(c_{i-1})) < 0$ and consequently $W(c_{i-1}) < 0$. Meanwhile, from 
		$u'(r_i) > 0$ and $v(r_i) > 0$, we get  $W(r_i) > 0$. However, by \eqref{deri w}, we have $W'(r) < 0$ for all $r \in [c_{i-1}, r_i]$, which is impossible. If $i$ is odd, a similar argument shows $W(c_{i-1}) > 0$ and $W(r_i)<0$, while $W'(r)> 0$ on $  [c_{i-1}, r_i]$, which is also impossible. Hence, $v(r)$ changes sign in $(c_{i-1}, r_i)$ for $i\in \lbr{2,\ldots,k}$, and also for $i=k+1$ when $u$ is a bound state. 
		
		Next, we prove \eqref{p1}.  On the interval $(c_{i-1},\tau_i)$, we have $|u|>1$, $uf(u)>0$, $f'(u)>0$. If $i$ is odd, then   $u>0$ and $u'<0$. By \eqref{q9}, we obtain $v'(c_{i-1}) < 0$ and $v(c_{i-1}) > 0$. Since $v$ does not change sign on $(c_{i-1}, \tau_i)$, it follows that $v > 0$ on this interval. Then from \eqref{equv2} we have $(r^{n-1}v')' = -r^{n-1}f'(u)v < 0$, which implies $r^{n-1}v' < c_{i-1}^{n-1}v'(c_{i-1}) < 0$ on $(c_{i-1}, \tau_i)$; hence $v' < 0$ on this interval.   Consequently, \eqref{p1} follows directly from    $u,v>0$ and  $u', v'<0$.  Similarly,  if $i$ is even, one finds $u, v < 0$ and $u', v' > 0$, so that \eqref{p1} remains valid.
		
		Finally, we prove that if \eqref{assumption2} holds, then 
		$\tau_i$ is the unique zero of $v$ on the interval $I_i$. Note that  $u'(\tau_i)v'(\tau_i)>0$. First, we consider the case where $i\in\lbr{2,\ldots,k-1}$, or $i=k$ when $u$ is a bound state. Suppose that $v$ has another zero in $(\tau_i,c_i]$, and  let $\tilde\tau_i$ be the first such zero.  If $\tilde\tau_i\in(\tau_i,c_i)$, then $u'(\tilde\tau_i)v'(\tilde\tau_i)<0$ and consequently $Q_1(\tilde\tau_i)=\tilde\tau_i^{n-1}u'(\tilde\tau_i)v'(\tilde\tau_i)<0$. If $\tilde\tau_i=c_i$, then $u'(\tilde\tau_i)v'(\tilde\tau_i)=0$ and $Q_1(\tilde\tau_i)=0$. Both possibilities contradict \eqref{assumption2}. Next, we consider the case $i=k$ when $u$ is not a bound state. Suppose that $v$ has another zero in $(\tau_i,z_k]$, and let $\tilde\tau_i$ be the first such zero. Then $u'(\tilde\tau_i)v'(\tilde\tau_i)<0$ and hence $Q_1(\tilde\tau_i)=\tilde\tau_i^{,n-1}u'(\tilde\tau_i)v'(\tilde\tau_i)<0$, which again contradicts \eqref{assumption2}. Finally, for the case $i=k+1$ with $u$ a bound state, we can similarly prove that $\tau_i$ is the unique zero of $v$ on $[c_k,\infty)$. The proof is completed.\end{proof}

	\begin{lemma}\label{lemma Qnpositive4}
		Assume that \eqref{assumption1} and \eqref{assumption2} hold. Then for $n\geq 4$ we have 
		\begin{equation}
			Q_n(r)>0, \quad \text{ on } ~[b_i,c_i],
		\end{equation}
		for $i\in\lbr{1,\ldots,k-1}$, and also for $i=k$ if $u$ is a bound state.
	\end{lemma}

	\begin{proof}
		For any $i\in\lbr{1,\ldots,k-1}$, and also for $i=k$ if $u$ is a bound state,  by Proposition \ref{prop phase1} and Lemma \ref{lemma sign},  the variation $v(r)$ has a unique zero $\tau_i$  on $\mbr{c_{i-1},c_i}$  with $\tau_i\in\sbr{c_{i-1},r_i}$. It is straightforward to show that $u'v\geq 0$ on $\mbr{\tau_i,c_i}$. Hence,   from \eqref{assumption2}  we  obtain
		\begin{equation}
			Q_n(r)\geq Q_1(r)>0, \quad \text{ on }~~\mbr{\tau_i,c_i}.
		\end{equation}
		Clearly, if $\tau_i\leq b_i$, then $Q_n(r)>0$ on $\mbr{b_i,c_i} $.   It remains to consider the case $b_i<\tau_i$.    In this situation, one can directly verify that $Q_n(c_i)=c_i^nf(u(c_i))v(c_i)>0$. Define
		\begin{equation}
			\theta_i=\inf\lbr{\theta: \theta\in \mbr{c_{i-1},c_i}, \quad Q_n(r)>0 \quad \text{ for all } ~r\in (\theta, c_i]}.
		\end{equation}
		Since $Q_n(\tau_i)>0$, we have $\theta_i<\tau_i$.  Obviously, 
		\begin{equation}
			Q_n(r)>0 ~~\text{ on } ~r\in [b_i, c_i] \quad \Leftrightarrow\quad  \theta_i<b_i.
		\end{equation}
		Suppose, for contradiction, that  $\theta_i\geq b_i$.  
		Hence, we have 
		\begin{equation}\label{o3}
			c_{i-1}  < b_i\leq \theta_i<\tau_i<r_i.
		\end{equation}
		By \eqref{p1} and \eqref{deri w}, we have $W'(r)>0$ in $\sbr{  c_{i-1},\tau_i}$. As $W(0)=0$ and  $W( c_{i-1})=- c_{i-1}^{n-1}f(u)v'>0$ for $i>1$, we have 
		\begin{equation}\label{u2}
			W(\theta_i)=\theta_i^{n-1} \mbr{f'(u)u'v-f(u)v'}  >W(c_{i-1} )\geq 0.
		\end{equation}
		Consequently, by \eqref{p1}, we have 
		\begin{equation}\label{o1}
			-\frac{f(u)v'}{f'(u)} >-u'v  \quad \text{at } ~~r=\theta_i.
		\end{equation}
		Moreover,  noting that  $Q_n(\theta_i)=0$, we obtain 
		\begin{equation}\label{o2}
			\theta_i \mbr{u'v'+f(u)v} =-2(n-1) u'v  \quad \text{at } ~~r=\theta_i.
		\end{equation}
		By \eqref{o3}, we obtain $f(u)v>0$ at  $r=\theta_i$. Combining \eqref{o2} with \eqref{o1} yields 
		\begin{equation}
			-\frac{f(u)v'}{f'(u)} >\frac{ \theta_i \mbr{u'v'+f(u)v}}{2(n-1)} >\frac{ \theta_i  u'v'  }{2(n-1)}\quad \text{at } ~~r=\theta_i.
		\end{equation}
		Therefore, using \eqref{p1} again, we obtain  \begin{equation}\label{o4}
			\frac{2(n-1)f(u) }{uf'(u)} =  \frac{2(n-1)f(u) v'}{f'(u)}\cdot \frac{1}{uv'}  > -\frac{ \theta_i  u'  }{u}  = \omega(\theta_i)\quad \text{at } ~~r=\theta_i,
		\end{equation}
		where $\omega$ is  defined in Lemma \ref{lemma positivefunction}. Since $uu'<0$ at $r=b_i$, it follows from   Lemma \ref{lemma positivefunction}  that  $\omega'(r)>0$ whenever $u(r)\neq 0$ and 
		\begin{equation}\label{u4}
			0<P(b_i)=b_i^nu'^2+(n-2)b_i^{n-1}uu'= b_i^{n-1}|uu'|\cdot\mbr{ \omega(b_i)-(n-2)}.
		\end{equation}
		Consequently, from \eqref{o3}, 
		\begin{equation}\label{o5}
			\omega(\theta_i)   > \omega(b_i)>n-2.
		\end{equation}
		Since $|u(r)|$ decreases in $\sbr{c_{i-1},\tau_i}$, we see from  \eqref{o3} that $|u(\theta_i)|\leq |u(b_i)|= \alpha_*=e^{1/2}$. Hence,   using \eqref{o4}, \eqref{o5} and the fact that 
		$f(u)=u\log u^2$, we obtain 
		\begin{equation}\label{u5}
			\frac{2(n-2)}{n}<  \log u^2(\theta_i)< \log u^2(b_i)=1,
		\end{equation}
		which is impossible if $n\geq 4$. The proof is completed.\end{proof}

	\begin{lemma}\label{lemma dimension3}
		Assume that \eqref{assumption1} and \eqref{assumption2} hold. Let $n=3$ and  suppose $\tau_i>b_i$. Then for every $i\in\lbr{1,\ldots,k-1}$, and also for $i=k$ if $u$ is a bound state, we have 
		\begin{equation}\label{o6}
			H_i:=\int_{b_i}^{\tau_i} u^2(r)\log \sbr{\abs{\frac{\tilde{a}}{u(r)}}^2} \frac{Q_3(r)}{ru'^2(r)} ~\mathrm{d}r>0,
		\end{equation}
		where $\tilde{a}$  is a constant such that  $\tilde{a}\geq  e^{1/2}$.
	\end{lemma}
	\begin{proof}
		The proof is inspired by that of \cite[Lemma 5.4]{TangInvent2026}. We first consider the case $\theta_i \leq b_i$. Then we have $Q_3(r) > 0$ on $\sbr{b_i,\tau_i}$, and consequently \eqref{o6} is trivially satisfied because  
		\begin{equation} \label{o9}
			1=|u(r_i)|<|u(r)| < |u(b_i)| = e^{1/2} \leq \tilde{a}  \quad \text{ on }~~\sbr{b_i,\tau_i}\subset \sbr{b_i,r_i}.
		\end{equation}
		It remains to consider the   case $ b_i<\theta_i<\tau_i$. We divide the proof into five steps.  
		
		\vskip0.04in
		{\it Step 1.} We   show that
		\begin{equation}\label{o8}
			Q_3'(r) >0 ~~ \text{ on } \sbr{b_i,\tau_i}.
		\end{equation}
		Observe that  
		\begin{equation}\label{u6}
			Q_3'(r)=r^{2}[3u'v'-f(u)v]=r^{2}u'v' \mbr{ 3- \frac{f(u)v}{u'v'} },
		\end{equation}
		and from \eqref{radialform} and \eqref{equv} we have 
		\begin{equation}\label{o10}
			\begin{aligned}
				\frac{\mathrm{d}}{\mathrm{d}r} \sbr{\frac{f(u)v}{u'v'}}&=\frac{\mbr{f'(u)u'v+f(u)v'}u'v'-f(u)v\mbr{u''v'+u'v''}}{\sbr{u'v'}^2}
				\\&=\frac{1}{u'v'}\mbr{f'(u)u'v+f(u)v'+\frac{ f(u)v}{ {u'v'} }\sbr{\frac{4}{r}u'v'+f(u)v'+f'(u)u'v}}\\
				&<\frac{1}{u'v'}\mbr{f'(u)u'v+f(u)v'+\frac{4}{r}f(u)v},
			\end{aligned}
		\end{equation}
		where the last inequality follows from \eqref{p1}.
		\vskip0.04in
		
		Since $M(c_{i-1})>0$ and $M'(r)=r^{2}\mbr{uf'(u)-f(u)}v=2r^{2}uv>0$ on $\sbr{c_{i-1},\tau_i}$, it follows that  $M (r)>M(c_{i-1})>0$ on $\sbr{b_i,\tau_i}\subset \sbr{c_{i-1},\tau_i}$. By the definition of $M$, we have $u'v>uv'$ on $\sbr{b_i,\tau_i} $, which implies that $ -rv'/v>-ru'/u =\omega(r)$ on this interval.  Similar to \eqref{o5}, we have $\omega(r)>\omega(b_i)>1$ on $\sbr{b_i,\tau_i} $.  
		Therefore, using \eqref{p1}, \eqref{o9}, and \eqref{o10}, we obtain
		\begin{equation}
			\begin{aligned}
				\frac{\mathrm{d}}{\mathrm{d}r} \sbr{\frac{f(u)v}{u'v'}} &<\frac{1}{u'v'}\mbr{\frac{f'(u)uv}{r }\sbr{ \frac{3f(u)}{uf'(u)} -\omega(r) }  +\frac{f(u)v}{r}\sbr{rv'/v+1 }}\\
				&<\frac{f'(u)uv}{ru'v'}\mbr{ \frac{3f(u)}{uf'(u)} -1 }\\
				&=\frac{f'(u)uv}{ru'v'}\cdot\frac{\log u^2 -2}{\log u^2 +2} <0, \quad \text{ on }~ \sbr{b_i,\tau_i}.
			\end{aligned}
		\end{equation}
		Hence, the function  ${f(u)v}/\sbr{{u'v'}} $ is strictly decreasing on $\sbr{b_i,\tau_i}$. We now proceed to prove \eqref{o8}. Assume by contradiction  that there exists $ \gamma_* \in \sbr{b_i,\tau_i}$ such that $Q_3'(\gamma_*)=0$, which implies that 
		\begin{equation}\label{io11}
			f(u)v=3u'v', \quad \text{ at }~r=\gamma_*.
		\end{equation}
		Moreover,  we have 
		$Q_3'(r)<0$ in  $\sbr{b_i,\gamma_* } $ and 
		$Q_3'(r)>0$ in  $\sbr{ \gamma_*,\tau_i } $. Consequently, for every    $r\in \sbr{b_i,\tau_i }$, we deduce from \eqref{io11} that
		\begin{equation}
			Q_3(r)\geq Q_3(\gamma_*)=\gamma_*^2\mbr{\gamma_*\sbr{u'v'+f(u)v}+4u'v }=4\gamma_*^2u'v(1+\gamma_*v'/v)>0,
		\end{equation}
		which contradicts   the assumption that $b_i<\theta_i<\tau_i$ and the definition of $\theta_i$. We thus conclude that \eqref{o8} is true.
		\vskip0.04in
		{\it Step 2.} We prove that $u''$  has a unique zero in $\sbr{c_{i-1},z_i}$. Without loss of generality, we focus on the case of odd $i$; the case of even $i$ follows a similar argument, with careful attention paid to the signs of the corresponding quantities. In this setting,   we have $u>0$ and $u'<0$ on $\sbr{c_{i-1},z_i}$, and 
		\begin{equation}
			u>1  ~~\text{ on } [c_{i-1},r_i),   \quad u(r_i)=1,\quad u<1  ~~\text{ on } \sbr{r_i,z_i},
		\end{equation}
		Thus,    $u''(c_{i-1})=-f(u(c_{i-1})  )<0$, and 
		\begin{equation}\label{o88}
			u''(r)=-2u'/r-f(u)>0, \quad \text{ in }~\sbr{r_i, z_i}.
		\end{equation}
		Hence, $u''$ has a zero in $\sbr{c_{i-1},r_i}$. Let $   \gamma_i \in \sbr{c_{i-1},r_i}$ be the first such  zero. Then 
		\begin{equation}
			u''(r)<0 ~~ \text{ in }~\sbr{c_{i-1}, \gamma_i }\quad \text{ and  }\quad  u''( \gamma_i )=0 ,
		\end{equation}
		and 
		\begin{equation}\label{op11}
			0\leq u'''( \gamma_i )=\frac{2}{  \gamma_i^{2}}u'(  \gamma_i )-f'(u(  \gamma_i ))u'(  \gamma_i )  .
		\end{equation}
		Since $u'(  \gamma_i )<0$ and $u''( \gamma_i )=0 $,   we see from \eqref{radialform} and \eqref{op11} that
		\begin{equation}
			-\frac{2}{  \gamma_i } u'=f(u), \quad   \frac{2}{  \gamma_i^{2}}\leq f'(u)\quad \text{ at  }~~ r=  \gamma_i.
		\end{equation}
		Therefore, 
		\begin{equation}\label{p99}
			\omega(  \gamma_i)=-\frac{ \gamma_i u'}{u}= -\frac{2}{  \gamma_i } u'\cdot \frac{  \gamma_i^{2}}{2} \cdot \frac{1}{u}\geq \frac{f(u)}{uf'(u)} \quad \text{ at  }~~ r=  \gamma_i.
		\end{equation}
		Define 
		\begin{equation} 
			\tilde{\omega} (r):=\frac{f(u)}{uf'(u)}=\frac{\log u^2}{\log u^2+2}.
		\end{equation}
		A direct calculation shows that  $\tilde{\omega} (r)$ is strictly decreasing on $\sbr{ \gamma_i,r_i}$. Moreover, by Lemma \ref{lemma positivefunction},  we have $\omega'>0$ on $\sbr{  \gamma_i,r_i}$. Then from \eqref{p99}, we obtain 
		\begin{equation}\label{o77}
			\omega(r)>\omega(  \gamma_i)>\tilde{\omega} (  \gamma_i)>\tilde{\omega}(r)\quad \text{ for } ~r\in \sbr{  \gamma_i,r_i},
		\end{equation}
		which implies that,  if $u''(\tilde{r})=0$ at some $\tilde{r}\in \sbr{ \gamma_i,r_i}$, then there holds $u'''(\tilde{r})>0$. 
		\vskip0.04in
		If $ u'''( \gamma_i )>0$, then $u''$ changes sign from negative to positive at  $ \gamma_i$, and stays positive in $\sbr{  \gamma_i,r_i}$.  Then from \eqref{o88},  $u''$  has a unique zero in $\sbr{c_{i-1},z_i}$.  It remains to consider the case 
		\begin{equation}
			u''(  \gamma_i )=u'''( \gamma_i )=0, \quad \text{ and  }\quad u''(r)<0 ~ \text{ in  } ~(  \gamma_i,  \gamma_i +\varepsilon]~ \text{ for a small }  \varepsilon>0.
		\end{equation}
		In such case, $u''$ has a zero in $\sbr{\gamma_i,r_i}$ since $u''(r_i)=-2u'(r_i)/r_i>0$. Then $u''$ has a negative minimum value $\tilde\gamma_i \in \sbr{\gamma_i,r_i}$ such that 
		\begin{equation}
			u''(\tilde\gamma_i)= -\frac{2}{\tilde\gamma_i}u'-f(u)<0\quad \text{ and }\quad 0=u'''(\tilde\gamma_i)>\frac{2}{  \gamma_i^{2}}u' -f'(u )u' .
		\end{equation}
		Since $u'(\tilde\gamma_i)<0$, we have $ -2u'/{\tilde\gamma_i}<f(u)$ and $2/{  \gamma_i^{2}}  >f'(u ) $. Then 
		\begin{equation}
			\omega(\tilde\gamma_i ) = -\frac{2}{  \tilde\gamma_i } u'\cdot \frac{  \tilde\gamma_i^{2}}{2} \cdot \frac{1}{u}<\tilde{\omega}  (\tilde\gamma_i ),
		\end{equation}
		which contradicts \eqref{o77}. Hence,  $u''$  has a unique zero in $\sbr{c_{i-1},z_i}$.
		\vskip 0.04in
		{\it Step 3.} We prove that $u'^2$ is strictly decreasing in $\sbr{b_i,\tau_i}$.   Observe that $$u'(z_i)u''(z_i) = -2u'^2(z_i)/z_i < 0,$$ and $u'$ keeps a constant sign on $(c_{i-1},z_i)$. It follows from Step 2 that   there exists $   \gamma_i \in \sbr{c_{i-1},r_i}$ such that 
		\begin{equation} 
			u'u''>0  ~ \text{ in  } ~ \sbr{c_{i-1},\gamma_i }, \quad \text{ and   }  \quad u'u''<0  ~ \text{ in  } ~  (\gamma_i,z_i ].
		\end{equation}
		Assume, to the contrary, that Step 3 does not hold. Then we must have  $b_i<\gamma_i$ 
		and  
		\begin{equation} \label{u1}
			u''v=-\frac{2}{b_i}u'v-f(u)v<0\quad \text{ at }  r=b_i.
		\end{equation}
		From Step 1 and $ b_i<\theta_i<\tau_i$, we know $Q_3(b_i)<Q_3(\theta_i)=0$. Combining this with  \eqref{defi of Qi} and \eqref{u1}, we obtain
		\begin{equation}\label{u3}
			u'v'+f(u)v<-\frac{4}{b_i} u'v<2f(u)v \quad \text{ at }  r=b_i.
		\end{equation}As in \eqref{u2}, we have
		\begin{equation}
			W(b_i)=b_i^{2} \mbr{f'(u)u'v-f(u)v'}  >W(c_{i-1} )\geq 0.
		\end{equation}
		Hence, at $r=b_i$,
		\begin{equation}
			\begin{aligned}
				-f(u)v'&>-f'(u)u'v= \frac{b_if'(u)}{4} \cdot \sbr{- \frac{4}{b_i} u'v}>\frac{b_if'(u)}{4}\cdot\mbr{ u'v'+f(u)v} .
			\end{aligned}
		\end{equation}
		From \eqref{u3} we then deduce, at $r=b_i$,   
		\begin{equation}
			\begin{aligned}
				2b_iu'v' <  b_i \mbr{ u'v'+f(u)v} <-\frac{4f(u)v'}{f'(u)}.
			\end{aligned}
		\end{equation}
		Combining this with \eqref{u4} gives 
		\begin{equation}
			1< \omega(b_i) =-  \frac{b_iu'}{u}|_{r=b_i}<\frac{2f(u)}{uf'(u)}|_{r=b_i}= \frac{2}{3},
		\end{equation}
		where we used  $|u(b_i)|=\alpha_*=e^{1/2}$.
		This is a contradiction.
		
		\vskip 0.04in
		{\it Step 4.} We prove that for each $\theta_i^-\in [b_i,\theta_i)$, there exists a unique $\theta_i^+\in (\theta_i,\tau_i)$ such that 
		\begin{equation}\label{o55}
			u(\theta_i^-)-u(\theta_i)= u(\theta_i)-u(\theta_i^+),
		\end{equation}
		By an argument similar to that used for \eqref{u5}, we  obtain $\log u^2(\theta_i)>\frac{2}{3}$.
		Hence, 
		\begin{equation}
			0<|u(b_i)|-|u (\theta_i)| <e^{1/2}-e^{1/3}.
		\end{equation}
		It follows from \eqref{u6} that 
		\begin{equation}\label{u8}
			\begin{aligned}
				Q_3(\tau_i) &=3\int_{\theta_i}^{\tau_i} r^2u'v' ~\mathrm{d}r- \int_{\theta_i}^{\tau_i} r^2f(u)v ~\mathrm{d}r<3\int_{\theta_i}^{\tau_i} r^2u'v' ~\mathrm{d}r
				<3  \int_{\theta_i}^{\tau_i} r^2|u'| |v'|~\mathrm{d}r.
			\end{aligned}
		\end{equation}
		Using \eqref{p1} and \eqref{equv2}, we find 
		\begin{equation}\label{o22}
			(r^2|v'|)'=r^2 f'(u)|v|>0, \quad \text{ on }~\sbr{\theta_i,\tau_i}.
		\end{equation}
		Then from \eqref{u8} we infer
		\begin{equation}
			\begin{aligned}
				\tau_i^3 |u'(\tau_i) | |v'(\tau_i) | =   Q_3(\tau_i)  
				<3\tau_i^2 |v'(\tau_i)| \int_{\theta_i}^{\tau_i} |u'| ~\mathrm{d}r=3\tau_i^2 |v'(\tau_i)| \cdot|u(\theta_i)-u(\tau_i)|.
			\end{aligned}
		\end{equation}
		Since $|u(\tau_i) |>1$ and $\omega(b_i)>1$, we deduce
		\begin{equation}
			3|u(\theta_i)-u(\tau_i)|> \tau_i|u'(\tau_i) |=|u(\tau_i) |\cdot\omega(\tau_i)>\omega(\tau_i)=\kappa_i\omega(b_i)>\kappa_i, 
		\end{equation}
		where  $\kappa_i= {\omega(\tau_i)}/{\omega(b_i)}>1$ because $\omega'>0$ in $(0,\infty)$. Moreover, note that $1<|u(\tau_i)|<|u(\theta_i)|<|u(b_i)|=e^{1/2}$, and then 
		\begin{equation}
			|u(\theta_i)-u(b_i)|<e^{1/2}-1-|u(\theta_i)-u(\tau_i)|<e^{1/2}-1-\kappa_i/3.
		\end{equation}
		Therefore,
		\begin{equation}\label{o44}
			\frac{u(\tau_i)-u(\theta_i)}{u(\theta_i)-u(b_i)}= \frac{|u(\tau_i)-u(\theta_i)|}{|u(\theta_i)-u(b_i)|}>\frac{\kappa_i}{3\sbr{e^{1/2}-1}-\kappa_i}>\frac{\kappa_i}{2-\kappa_i}>\kappa_i^2>1.
		\end{equation}
		Then for each $\theta_i^-\in [b_i,\theta_i)$, we can find  $\theta_i^+\in (\theta_i,\tau_i)$ such that 
		\begin{equation}
			u(\theta_i^-)-u(\theta_i)= u(\theta_i)-u(\theta_i^+).
		\end{equation}
		Uniqueness follows directly from the monotonicity of $u$ on $\sbr{b_i,\tau_i}$.
		
		\vskip 0.04in
		{\it Step 5.}  We proceed to establish \eqref{o6}. Given that  $|u(r)|<|u(b_i)|<\tilde{a}$, $Q_3(\theta_i)=0$ and $Q_3'(r)>0$ holds for all  $r\in \sbr{b_i,\tau
			_i}$, it follows that
		\begin{equation}\label{y1}
			\begin{aligned}
				H_i&=\int_{b_i}^{\tau_i}\log \sbr{\abs{\frac{\tilde{a}}{u(r)}}^2} \frac{rQ_3(r)}{\omega^2(r)} ~\mathrm{d}r \\&=\int_{b_i}^{\theta_i}\log \sbr{\abs{\frac{\tilde{a}}{u(r)}}^2} \frac{rQ_3(r)}{\omega^2(r)} ~\mathrm{d}r+\int_{\theta_i}^{\tau_i}\log \sbr{\abs{\frac{\tilde{a}}{u(r)}}^2} \frac{rQ_3(r)}{\omega^2(r)} ~\mathrm{d}r\\&>\theta_i \log \sbr{\abs{\frac{\tilde{a}}{u(\theta_i)}}^2}\int_{b_i}^{\tau_i}\frac{Q_3(r)}{\omega^2(r)} ~\mathrm{d}r.
			\end{aligned}
		\end{equation}
		Lemma \ref{lemma positivefunction} implies that $\omega(r)$ is increasing on $\sbr{\theta_i,\tau_i}$. Therefore, we have
		\begin{equation}\label{y2}
			\begin{aligned}
				\int_{b_i}^{\tau_i}\frac{Q_3(r)}{\omega^2(r)} ~\mathrm{d}r&=\int_{\theta_i}^{\tau_i} \frac{Q_3(r)}{\omega^2(r)} ~\mathrm{d}r-\int_{b_i}^{\theta_i}  \frac{|Q_3(r)|}{\omega^2(r)} ~\mathrm{d}r \\
				&>\frac{1}{\omega^2(\tau_i)}\int_{\theta_i}^{\tau_i} {Q_3(r)}~\mathrm{d}r-\frac{1}{\omega^2(b_i)}\int_{b_i}^{\theta_i} {|Q_3(r)|} ~\mathrm{d}r\\
				&=\frac{1}{\omega^2(\tau_i)}\mbr{\int_{\theta_i}^{\tau_i} {Q_3(r)}~\mathrm{d}r-\kappa_i^2\int_{b_i}^{\theta_i} {|Q_3(r)|} ~\mathrm{d}r }.
			\end{aligned}
		\end{equation}
		Since $u(r)$ is monotone in $(b_i,\tau_i)$, we denote its inverse by  $r_u$, that is  $u(r_u)=u$. For each $\mu$ between $u(\theta_i)$ and $u(b_i)$, there exists a unique $s\in \mbr{0,1}$ such that 
		\begin{equation}\label{o33}
			\mu=u(r_\mu)=u(\theta_i)+ [u(b_i)-u(\theta_i)] s.
		\end{equation}
		This gives a bijection between   $\mu$ and $s$. We set $\tilde r_s:=r_{\mu(s)} $.  From Step 3,   $|u'(r)|$ is strictly decreasing in $(b_i,\tau_i)$. 
		Therefore,
		\begin{equation}\label{y3}
			\begin{aligned}
				\int_{b_i}^{\theta_i} {|Q_3(r)|} ~\mathrm{d}r &< \frac{1}{u'(\theta_i) }\int_{b_i}^{\theta_i} {|Q_3(r)|} u'(r) ~\mathrm{d}r =\frac{1}{u'(\theta_i) }\int_{u(b_i)}^{u(\theta_i)} {|Q_3(r_u)|} ~\mathrm{d}u\\&=\frac{u(\theta_i)-u(b_i)}{u'(\theta_i) }\int_{0}^{1} {|Q_3(\tilde r_s)|} ~\mathrm{d}s.
			\end{aligned}
		\end{equation}
		Analogously, for each  $\mu$ between $u(\theta_i)$ and $u(\tau_i)$, we can uniquely choose $s\in \mbr{0,1}$  with
		\begin{equation}\label{o222}
			\mu=u(r_\mu)=u(\theta_i)+ [u(\tau_i)-u(\theta_i)] s.
		\end{equation}
		In such case, we define $\hat r_s:=r_{\mu(s)} $. Then 
		\begin{equation}\label{y4}
			\begin{aligned}
				\int_{\theta_i}^{\tau_i} {Q_3(r)}~\mathrm{d} r &>\frac{1}{u'(\theta_i) }\int_{\theta_i}^{\tau_i} {Q_3(r)} u'(r) ~\mathrm{d}r =\frac{1}{u'(\theta_i) }\int_{u(\theta_i)}^{u(\tau_i)} {Q_3(r_u)} ~\mathrm{d}u\\&=\frac{u(\tau_i)-u(\theta_i)}{u'(\theta_i) }\int_{0}^{1} {Q_3(\hat r_s)} ~\mathrm{d}s.
			\end{aligned}
		\end{equation}
		For any  given $s\in \mbr{0,1}$, we write $\theta_i^-=\tilde r_s\in  [b_i,\theta_i)$. By Step 4, there exists a unique $\theta_i^+\in  \sbr{\theta_i,\tau_i}$  such that 
		\begin{equation}
			|u(\theta_i)-u(\theta_i^+)|=| u(\theta_i^-)-u(\theta_i)|.
		\end{equation}
		From \eqref{u6} and \eqref{o22},  we obtain 
		\begin{equation}
			\begin{aligned}
				Q_3(\theta_i^+)&=\int_{\theta_i}^{\theta_i^+}r^{2}u'v' \mbr{ 3- \frac{f(u)v}{u'v'} }~\mathrm{d}r>\mbr{ 3- \frac{f(u)v}{u'v'} |_{r=\theta_i} }  \int_{\theta_i}^{\theta_i^+}r^{2}|u'||v'|~\mathrm{d}r\\
				&>\mbr{ 3- \frac{f(u)v}{u'v'} |_{r=\theta_i} } \theta_i^2\cdot|v'(\theta_i)|\cdot |u(\theta_i)-u(\theta_i^+)|,
			\end{aligned}
		\end{equation}
		because ${f(u)v}/\sbr{u'v'}<3$  and  this ratio is decreasing  on $(b_i,\tau_i)$. Similarly, 
		\begin{equation}
			\begin{aligned}
				| Q_3(\theta_i^-)|&=\int_{\theta_i^-}^{\theta_i}r^{2}u'v' \mbr{ 3- \frac{f(u)v}{u'v'} }~\mathrm{d}r<\mbr{ 3- \frac{f(u)v}{u'v'} |_{r=\theta_i} }  \int_{\theta_i^-}^{\theta_i}r^{2}|u'||v'|~\mathrm{d}r\\
				&<\mbr{ 3- \frac{f(u)v}{u'v'} |_{r=\theta_i} } \theta_i^2\cdot|v'(\theta_i)|\cdot |u(\theta_i)-u(\theta_i^-)|.
			\end{aligned}
		\end{equation}
		Therefore, 
		\begin{equation}\label{o111}
			| Q_3(\theta_i^-)|< Q_3(\theta_i^+).
		\end{equation}
		Moreover, using \eqref{o44}, \eqref{o33} and \eqref{o222}, we have
		\begin{equation}
			\begin{aligned}
				&|u(\theta_i)-u(\theta_i^+)|=| u(\theta_i^-)-u(\theta_i)|= |u(\tilde r_s)-u(\theta_i)|\\
				&=|u(b_i)-u(\theta_i)| s<|u(\tau_i)-u(\theta_i)|s=|u(\hat r_s)-u(\theta_i)|.
			\end{aligned}   
		\end{equation}
		Therefore, $\hat r_s>\theta_i^+$. From Step 1 and \eqref{o111}, we obtain
		\begin{equation}\label{y5}
			Q_3(\hat r_s)>Q_3(\theta_i^+)> | Q_3(\theta_i^-)|= | Q_3(\tilde  r_s)|.
		\end{equation}
		Therefore, from \eqref{o44}, \eqref{y3}, \eqref{y4} and \eqref{y5}, we deduce
		\begin{equation}
			\begin{aligned}
				&\int_{\theta_i}^{\tau_i} {Q_3(r)}~\mathrm{d}r-\kappa_i^2\int_{b_i}^{\theta_i} {|Q_3(r)|} ~\mathrm{d}r \\
				& >\frac{u(\tau_i)-u(\theta_i)}{u'(\theta_i) }\int_{0}^{1} {Q_3(\hat r_s)} ~\mathrm{d}s-\frac{u(\theta_i)-u(b_i)}{u'(\theta_i) }\kappa_i^2\int_{0}^{1} {|Q_3(\tilde r_s)|} ~\mathrm{d}s\\
				& >\frac{u(\theta_i)-u(b_i)}{u'(\theta_i) } \mbr{\frac{u(\tau_i)-u(\theta_i)}{u(\theta_i)-u(b_i) } -\kappa_i^2}\int_{0}^{1} {|Q_3(\tilde r_s)|} ~\mathrm{d}s>0.
			\end{aligned}
		\end{equation}
		Finally, from \eqref{o9}, \eqref{y1} and \eqref{y2} we conclude  $H_i>0$. This completes the proof.
	\end{proof}
	
	\begin{lemma}
		We have \begin{equation}\label{T2B0connection}
			T_2(r)=B_0(r)+\frac{F(u)}{uu'}T_2'(r).
		\end{equation} 
	\end{lemma}
	\begin{proof}
		It follows from \eqref{defi of deriQMT} that 
		\begin{equation}
			M(r)=\frac{u}{2u'}\mbr{ 2r^{n-1}uv-T_2'(r)}.
		\end{equation}
		Since $g_2(u)=\log u^2-1={2F(u)}/{u^2}$, we have 
		\begin{equation}
			\begin{aligned}
				T_2(r)&=Q(r)-g_2(u)M(r)=Q(r)-\frac{2F(u)}{u^2} \cdot \frac{u}{2u'}\mbr{ 2r^{n-1}uv-T_2'(r)}\\
				&=Q(r)-2F(u)\cdot \frac{r^{n-1}v}{u'}+\frac{F(u)}{uu'}T_2'(r)\\
				&=B_0(r)+\frac{F(u)}{uu'}T_2'(r),
			\end{aligned}
		\end{equation}
		where we use the definition  $B_a(r)$ \eqref{defi of Ba} with $a=0$. This completes the proof.
	\end{proof}
	\begin{lemma}\label{lemma T2}
		Assume that \eqref{assumption1} and \eqref{assumption2} hold. We further assume that    for  any given $i\in\lbr{2,\ldots,k-1}$, and also for $i=k$ if $u$ is a bound state,  
		\begin{equation}\label{assumption3}
			Q(r)>0\quad  \text{ in }~(c_{i-1},b_i]\cup \mbr{\bar b_i,c_i}, \quad Q(\bar b_i)>Q(b_i),  
		\end{equation}
		and
		\begin{equation}\label{assumption4}
			T_2(r)>0  \quad  \text{ in }~\mbr{c_{i-1},b_i}, \quad Q_2(r)>0  \quad  \text{ in }~\mbr{c_{i-1},c_i},
		\end{equation}
		then we have
		\begin{equation}\label{u33}
			T_2(r)>0  \quad  \text{ in }~\mbr{\bar b_i,c_i}.
		\end{equation}
		In particular, $ T_2(r)>0$ in $\mbr{\bar b_1,c_1}$.
	\end{lemma}

	\begin{proof}
		
		For any given   $i\in\lbr{2,\ldots,k-1}$, and also for $i=k$ if $u$ is a bound state, Lemma \ref{lemma sign} implies that $v(r)$ has a unique zero $\tau_i\in(c_{i-1},r_i)$ in the interval $\mbr{c_{i-1},c_i}$. Note that by \eqref{defi of deriQMT}, $T_2'(c_i)=2c_i^{n-1}u(c_i)v(c_i)>0$. Consequently, if $T_2$ has no critical points in $\sbr{\bar{b}_i,c_i}$, then on this interval  $T_2'(r)>0$  and $T_2(r)>T_2(\bar{b}_i)=Q(\bar{b}_i)>0$. Hence, \eqref{u33} holds.
		\vskip0.04in
		We now turn to the case  where $T_2$   has critical points in    $\sbr{\bar{b}_i,c_i}$. Let $\bar t_i$   be an arbitrary critical point  within this interval. Using \eqref{T2B0connection} and \eqref{defi of Ba} with $a=0$, we obtain
		\begin{equation}
			T_2'(r)=-2F(u)\cdot \frac{Q_n(r)}{ru'^2}+\sbr{\frac{F(u)}{uu'}}'T_2'(r)+\frac{F(u)}{uu'}T_2''(r).
		\end{equation}
		Since $T_2'(\bar t_i)=0$, it follows that 
		\begin{equation}
			T_2''(\bar t_i)=   \frac{ 2u   }{\bar t_iu' } \cdot Q_n(\bar t_i)  , \quad \text{ at } r=\bar t_i.
		\end{equation} 
		Because  $u'v>0$ and $u/u'>0$ in $\sbr{\bar{b}_i,c_i}$, we have $Q_n(\bar t_i) >Q (\bar t_i)>0$,   and therefore $T_2''(\bar t_i)>0$.  Hence, $\bar{t}_i$ is a   minimum point of $T_2$ and, in fact, the unique  critical
		point of $T_2$ in $\sbr{\bar{b}_i,c_i}$.  Therefore, 
		\begin{equation}\label{y11}
			T_2'(r)<0 \text{ in } \sbr{\bar{b}_i,\bar{t}_i }, \quad  T_2'(\bar{t}_i)=0, \quad \text{ and }\quad T_2'(r)>0 \text{ in } \sbr{ \bar{t}_i,c_i }.
		\end{equation}
		and 
		\begin{equation}\label{p11}
			T_2(\bar{t}_i)=\min_{r\in \mbr{\bar{b}_i,c_i}} T_2(r).
		\end{equation}
		Since $|u(c_{i-1})|>|u(c_i)|$ and $|u(r)|$ is monotone in $(c_{i-1},c_i)$, we can find a unique $t_i\in \sbr{c_{i-1},b_i}$  such that $$u(t_i)=-u(\bar{t}_i).$$ Let 
		\begin{equation}
			a_i=g_2(u(t_i))=g_2(u(\bar t_i))=\log u^2(t_i)-1.
		\end{equation}
		Then 
		\begin{equation}
			\begin{aligned}
				F_{a_i}(u(t_i))&=F(u(t_i))-\frac{a_i}{2}\mbr{u(t_i)f(u(t_i))-2F(u(t_i))}\\
				&=\frac{1}{2}u^2(t_i)\mbr{\log u^2(t_i)-1-a_i }=0.
			\end{aligned}
		\end{equation}
		Therefore, 
		\begin{equation} \label{u55}
			\begin{aligned}
				T_2(\bar t_i)-  T_2( t_i)&=B_{a_i}(\bar t_i)-B_{a_i}(  t_i)= \int_{t_i}^{\bar t_i}B_{a_i}'(r) ~\mathrm{d} r = \int_{t_i}^{\bar t_i} -2F_{a_i}(u)\cdot \frac{Q_n(r)}{r|u'(r)|^2}  ~\mathrm{d} r .
			\end{aligned}
		\end{equation}
		Set 
		\begin{equation}
			D_i:=\int_{t_i}^{\bar t_i} -2F_{a_i}(u)\cdot \frac{Q_n(r)}{r|u'(r)|^2}  ~\mathrm{d} r=\int_{t_i}^{\bar t_i} u^2(r)\log \sbr{\abs{\frac{u(t_i)}{u(r)}}^2} \frac{Q_n(r)}{r|u'(r)|^2} ~\mathrm{d}r.
		\end{equation}
		It suffices to show  that $D_i>0$. Indeed, from \eqref{p11} and \eqref{assumption4}, it follows that for any $r \in \mbr{b_i,t_i}$, 
		\begin{equation}
			T_2( r)\geq     T_2(\bar t_i)= T_2( t_i)+  D_i>0,
		\end{equation}
		which is the desired result \eqref{u33}. Note that if $t_i\geq \theta_i$, then we have $Q_n(r)>0$ on $\sbr{t_i, \bar t_i}$. Since $|u(r)|<|u(t_i)|$ on the same interval,  it follows that $D_i>0$. 
		
		\vskip0.08in
		The case $t_i<\theta_i$ is more delicate. In this situation, from  
		$Q_n(\tau_i)>0$ we deduce that  
		\begin{equation}\label{y22}
			c_{i-1}<   t_i<\theta_i< \tau_i, \quad \text{and} \quad t_i<b_i, \quad \bar b_i<\bar t_i.
		\end{equation}
		For any $\mu$ in the interval $ [\alpha_*, |u(t_i)|)$, we define $r_{i\mu}\in (t_i,b_i]$ and $\bar r_{i\mu}\in [\bar b_i, \bar t_i)$ such that $u(r_{i\mu})=u(\bar r_{i\mu})=\mu$. Then by Lemma \ref{lemma propertyE} (ii) we have
		\begin{equation}\label{j6}
			|u'(r_{i\mu})|>|u'(\bar r_{i\mu})|.
		\end{equation}Inspired by  \cite[Lemma 5.6]{TangInvent2026}, we claim    that 
		\begin{equation}\label{u777}
			\frac{Q_n(\bar r_{i\mu})}{\bar r_{i\mu}  |u'(\bar r_{i\mu})|^3}+\frac{Q_n(r_{i\mu})}{r_{i\mu}  |u'(r_{i\mu})|^3}>0, \quad  \text{ for   } ~\mu\in R_i:=\sbr{\max \lbr{|u(\theta_i)|, |u(b_i)| }, |u(t_i)|}.
		\end{equation}
		Indeed, since  $R_i\subseteq \sbr{|u(b_i)|,|u(t_i)|}=\sbr{\alpha_*,|u(t_i)|}$ and all the assumptions in the proof of  \cite[Lemma 5.5]{TangInvent2026} remain valid (note that the proof of that lemma does not depend on the condition (C3)),  then by the same argument as in that lemma we obtain the following comparison result
		\begin{equation} \label{j1}
			\frac{Q(\bar r_{i\mu})}{\bar r_{i\mu}^{n-1} |u'(\bar r_{i\mu})|}>\frac{Q(r_{i\mu})}{r_{i\mu}^{n-1} |u'(r_{i\mu})|} ,\quad \text{ for  } ~\mu\in R_i.
		\end{equation}
		It follows from Lemma \ref{lemma positivefunction} that $P(r)/r^n$ is universally decreasing. Hence, we deduce from \eqref{j6}, \eqref{j1} that for $\mu\in R_i$, 
		\begin{equation}\label{j8}
			\begin{aligned}
				& \frac{  \bar r_{i\mu}^{n-2}Q(\bar r_{i\mu})  }{P(\bar r_{i\mu})  |u'(\bar r_{i\mu})|^2}= \frac{Q(\bar r_{i\mu})}{\bar r_{i\mu}^{n-1} |u'(\bar r_{i\mu})|}\cdot  \frac{ \bar r_{i\mu}^{n-3}  }{|u'(\bar r_{i\mu})|} \cdot\frac{\bar r_{i\mu}^{n}}{P(\bar r_{i\mu})  }\\
				>&\frac{Q(r_{i\mu})}{r_{i\mu}^{n-1} |u'(r_{i\mu})|} \cdot\frac{  r_{i\mu}^{n-3}  }{|u'( r_{i\mu})|} \cdot\frac{ r_{i\mu}^{n}}{P( r_{i\mu})  } =\frac{   r_{i\mu}^{n-2}Q( r_{i\mu})  }{P( r_{i\mu})  |u'(r_{i\mu})|^2}.
			\end{aligned} 
		\end{equation}By \eqref{y22}, we have $u'v<0$ on $\sbr{t_i,\min\lbr{\theta_i,b_i}}\subset \sbr{c_{i-1},\tau_i }$ and $u'v>0$ on $\sbr{\bar b_i,\bar t_i}\subset \sbr{\tau_i, c_i}$. It then follows from \eqref{assumption4} that
		\begin{equation}\label{j2}
			-Q_n(r)=-Q_2(r)-(n-2)r^{n-1}u'v\leq nr^{n-1}|u'v|\quad \text{on } \sbr{t_i,\min\lbr{\theta_i,b_i}},
		\end{equation}
		while from \eqref{assumption3} we obtain
		\begin{equation}\label{j3}
			Q_n(r)=Q(r)+nr^{n-1}u'v>nr^{n-1}u'v=nr^{n-1}|u'v|\quad \text{on } \sbr{\bar b_i,\bar t_i}.
		\end{equation}
		Moreover,  we claim that 
		\begin{equation}\label{Connection2}
			|v|<\frac{Q(r)}{P(r)}|u| \quad \text{in } ~\sbr{t_i,\tau_i}, \quad \quad  |v|>\frac{Q(r)}{P(r)}|u| \quad \text{in } ~\sbr{\bar b_i,\bar t_i}.
		\end{equation}
		In fact, by a direct calculation, we have the following connection identity  
		\begin{equation}\label{ConnectionPQ}
			Q(r)-P(r)\cdot \frac{v}{u}=\omega(r) M(r)+r^nuv= \omega(r)\mbr{ M(r)- \frac{r^{n-1}u^2v}{u'}}.
		\end{equation}
		Since $uv$, $\omega$, $M>0$  in $(c_{i-1},\tau_i)$, then it follows from \eqref{ConnectionPQ}  that $Q(r)>P(r)v/u$; thus the first part of \eqref{Connection2} holds. On $\sbr{\bar b_i,\bar t_i}$,  by using \eqref{ConnectionPQ}, \eqref{defi of deriQMT}  and  \eqref{y11}, we obtain
		\begin{equation}
			Q(r)-P(r)\cdot\frac{v}{u}= \frac{1}{2}rT_2'(r)<0 \quad \text{for } r\in\sbr{\bar b_i,\bar t_i},
		\end{equation}
		which implies that the second part of \eqref{Connection2} holds, since $uv>0$ on this interval. 
		\vskip0.08in
		
		Note that for  $\mu\in R_i$, we have $ r_{i\mu}\in \sbr{t_i,\min\lbr{\theta_i,b_i}}$ and $\bar r_{i\mu}\in \sbr{\bar b_i,\bar t_i}$. Then we deduce from  \eqref{j8},  \eqref{j2}, \eqref{j3} and  \eqref{Connection2}   that  for $\mu\in R_i$, 
		\begin{equation}
			\begin{aligned}
				& \frac{Q_n(\bar r_{i\mu})}{\bar r_{i\mu}  |u'(\bar r_{i\mu})|^3} > \frac{ n\bar r_{i\mu}^{n-2}| v(\bar r_{i\mu})|}{   |u'(\bar r_{i\mu})|^2}  >  \frac{ n\bar r_{i\mu}^{n-2}Q(\bar r_{i\mu}) |u(\bar r_{i\mu})|}{P(\bar r_{i\mu})  |u'(\bar r_{i\mu})|^2} ~~~ \quad \text{ by }~  \eqref{j3},~\eqref{Connection2}\\
				=& \frac{ n\mu\bar r_{i\mu}^{n-2}Q(\bar r_{i\mu})  }{P(\bar r_{i\mu})  |u'(\bar r_{i\mu})|^2}>\frac{   n\mu r_{i\mu}^{n-2}Q( r_{i\mu})  }{P( r_{i\mu})  |u'(r_{i\mu})|^2}
				>\frac{   n\mu r_{i\mu}^{n-2}  |v(r_{i\mu})| }{   |u(r_{i\mu})||u'(r_{i\mu})|^2} \quad \text{ by }~  \eqref{j8},~\eqref{Connection2}\\=&\frac{   n r_{i\mu}^{n-2}  |v(r_{i\mu})| }{    |u'(r_{i\mu})|^2}>- \frac{    Q_n(r_{i\mu})   }{  r_{i\mu}  |u'(r_{i\mu})|^3} \quad \quad \quad \quad \quad \quad \quad \quad \quad \quad \quad~  \text{ by }~  \eqref{j2}.
			\end{aligned}
		\end{equation}
		Therefore, the relation \eqref{u777} holds. Now we are ready to prove $D_i>0$  when  $t_i<\theta_i$.     The proof is divided  into   two  cases  $\theta_i\leq b_i$ and $\theta_i>b_i$.
		
		\vskip 0.08in  
		\textit{Case 1.} If $\theta_i\leq b_i$, then we choose $\bar \theta_i\in \sbr{z_i,c_i}$   such that   $u(\bar \theta_i)=-u( \theta_i)$.
		Since $Q_n(r)>0$ in $(\theta_i,c_i]$ and $|u(r)|< |u(t_i)|$ in $\sbr{t_i,\bar t_i}$, we obtain 
		\begin{equation}\label{u666}
			\int_{\theta_i}^{\bar \theta_i} W_i(u)\cdot \frac{Q_n(r)}{r|u'(r)|^2} ~\mathrm{d}r >0, 
		\end{equation}
		where 
		\begin{equation}
			W_i(u(r)):=u^2(r)\log \sbr{\abs{\frac{u(t_i)}{u(r)}}^2}>0 ~~\text{ in } \sbr{t_i,\bar t_i}.
		\end{equation}
		If $i$ is odd, then $u>0$, $u'<0$ in $\sbr{t_i,\theta_i}$, and $u<0$, $u'<0$ in $\sbr{\bar \theta_i, \bar t_i}$. Therefore, since $|u(\theta_i)|\geq |u(b_i)|$, we deduce   from \eqref{u777} and \eqref{u666} that
		\begin{equation} \label{u44}
			\begin{aligned}
				&  D_i >\int_{t_i}^{\theta_i} W_i(u(r))\cdot \frac{Q_n(r)}{r|u'(r)|^2} ~\mathrm{d}r+\int_{\bar \theta_i}^{\bar t_i}W_i(u(r))\cdot \frac{Q_n(r)}{r|u'(r)|^2} ~\mathrm{d}r\\
				&=\int_{u(t_i)}^{u( \theta_i)} W_i(u )\cdot \frac{Q_n(r_{iu})}{r_{iu}|u'(r_{iu})|^2 u'(r_{iu})}     \mathrm{d}u+\int_{u( \bar \theta_i)}^{u( \bar t_i)}W_i(u)\cdot \frac{Q_n(\bar r_{i(-u)})}{\bar r_{i(-u)}|u'(\bar r_{i(-u)})|^2u'(\bar r_{i(-u)})}  \mathrm{d}u\\
				&=\int_{|u( \theta_i)|}^{|u(t_i)|} W_i(u )\cdot \frac{Q_n(r_{iu})}{r_{iu}|u'(r_{iu})|^3  }    ~\mathrm{d}u+\int_{|u( \theta_i)|}^{|u(t_i)|} W_i(u)\cdot \frac{Q_n(\bar r_{iu})}{\bar r_{iu}|u'(\bar r_{iu})|^3} ~\mathrm{d}u\\
				&=\int_{|u(t_i)|}^{|u( \theta_i)|} W_i(u )\cdot \sbr{ \frac{Q_n(\bar r_{i\mu})}{\bar r_{i\mu}  |u'(\bar r_{i\mu})|^3}+\frac{Q_n(r_{i\mu})}{r_{i\mu}  |u'(r_{i\mu})|^3} }~\mathrm{d}u>0.
			\end{aligned}
		\end{equation}
		If   $i$ is even, then $u<0$, $u'>0$ in $\sbr{t_i,\theta_i}$, while $u>0$, $u'>0$ in $\sbr{\bar \theta_i, \bar t_i}$.  Then, similarly to \eqref{u44}, we also obtain  $D_i>0$. 
		
		\vskip 0.08in
		\textit{Case 2.} If $\theta_i>b_i$, then $Q_n(b_i)<0$. According to Lemma \ref{lemma Qnpositive4}, this situation is possible only for $n=3$.  In such case, by Lemma \ref{lemma dimension3} and $Q_3(r)>0$ on $\sbr{\tau_i,\bar b_i}\subset \sbr{\theta_i,c_i}$, we obtain 
		\begin{equation}
			\int_{b_i}^{\tau_i} W_i(u ) \cdot\frac{Q_n(r)}{r|u'(r)|^2} ~\mathrm{d}r>0, \quad \int_{\tau_i}^{\bar b_i} W_i(u ) \cdot\frac{Q_n(r)}{r|u'(r)|^2} ~\mathrm{d}r>0.
		\end{equation}
		Because 
		$|u(\theta_i)|< |u(b_i)|$, an argument similar to Case 1 together with \eqref{u777} gives
		\begin{equation}  
			\begin{aligned}
				&  D_i >\int_{t_i}^{b_i} W_i(u(r))\cdot \frac{Q_n(r)}{r|u'(r)|^2} ~\mathrm{d}r+\int_{\bar b_i}^{\bar t_i}W_i(u(r))\cdot \frac{Q_n(r)}{r|u'(r)|^2} ~\mathrm{d}r\\
				&=\int_{|u(t_i)|}^{|u(b_i)|} W_i(u )\cdot \sbr{ \frac{Q_n(\bar r_{i\mu})}{\bar r_{i\mu}  |u'(\bar r_{i\mu})|^3}+\frac{Q_n(r_{i\mu})}{r_{i\mu}  |u'(r_{i\mu})|^3} }~\mathrm{d}u>0.
			\end{aligned}
		\end{equation}
		Consequently, we have proved that   \eqref{u33} holds for any given   $i\in\lbr{2,\ldots,k-1}$, and also for $i=k$ if $u$ is a bound state.
		
		Finally, we consider the case $i=1$. In such case,   Proposition \ref{prop phase1} and the identity $g_1(u)=g_2(u)+1$ imply $T_2(r)>T_1(r)>0$ on $\sbr{0,z_1}$. We now prove 
		\begin{equation}\label{y66}
			Q(\bar b_1)>Q(b_1).
		\end{equation}
		Indeed, if $b_1\geq \theta_1$, then on $\sbr{b_1,\bar b_1}$ we have $F(u)<0$ and  $Q_n(r)>0$; thus   $ B_0'(r)=-2F(u(r))\cdot  {Q_n(r)}/\sbr{ru'^2}>0$,and consequently 
		$Q(\bar b_1)=B_0(\bar b_1)>B_0( b_1)=Q( b_1)$. If instead $b_1 \leq \theta_1$, then $b_1<\tau_1$.  Since $u'v>0$ in $\sbr{\tau_1,\bar b_1}$, we deduce from Proposition \ref{prop phase1} that   $Q_n(r)>Q_1(r)>0$ in $\sbr{\tau_1,\bar b_1}$.  Consequently, $Q(\bar b_1)=B_0(\bar b_1) >B_0(\tau_1)=Q(\tau_1)$. Because  $uv>0$ and $ u >1$ in $\sbr{b_1,\tau_1}$, we see from   \eqref{defi of deriQMT}  that $Q'(r)=2r^{n-1} \sbr{u\log u^2 }v>0$; therefore, $Q(\bar b_1)>Q(\tau_1)>Q(b_1)$. Hence, \eqref{y66} holds. Note that Proposition \ref{prop phase1} also shows that all the assumptions made for  $i>1$ are all satisfied for $i=1$ as well, except that $M$, $Q$, $Q_1$, $Q_2$ and $T_2$ all vanish at $c_0=0$,   but when $i>1$ we need them to be positive at $c_{i-1}$.   Nevertheless, the arguments for $i>1$   still hold for $i=1$ with only minor modifications; we omit the details here.   The proof is completed.\end{proof}

	With Lemmas \ref{lemma sign}-\ref{lemma T2} at hand, we are now ready to prove Proposition \ref{prop phasei}.
	\begin{proof}[\bf Proof of Proposition \ref{prop phasei}]

		By Proposition \ref{prop phase1} and Lemma \ref{lemma T2}, we have $Q(c_1), M(c_1), T_2(c_1)>0$. 
		We now prove the proposition by induction under the assumption that for some $i\in\{2,\dots,k\}$  or $i=k+1$ if $u$ is a bound state,  
		\begin{equation}\label{sq0}
			Q(c_{i-1}), ~M(c_{i-1}), ~T_2(c_{i-1})>0 .  
		\end{equation}
		According to Lemma \ref{lemma sign}, $v(r)$ changes sign in $(c_{i-1},r_i)$. Let $\tau_i$ be the first zero in this interval. From \eqref{p1} we obtain  $uv,~f(u)v>0$ in $(c_{i-1},\tau_i)$. Hence, $Q'(r)=2r^{n-1}f(u)v>0$,  and $  M'(r)=2r^{n-1}uv>0 $ in $(c_{i-1},\tau_i]$. Then 
		\begin{equation}\label{sq2}
			\begin{aligned} 
				&Q(r)>Q(c_{i-1})>0,\quad M(r)>M(c_{i-1})>0,\quad r\in(c_{i-1},\tau_i].
			\end{aligned}
		\end{equation}
		We split the proof into four steps.  
		
		\vskip 0.1in
		{\it Step 1.} We   prove 
		\begin{equation}\label{sq1-1}
			T_2'(r),~T_2(r)>0,~~~~r\in[c_{i-1},\max\{\tau_i,b_i\}],
		\end{equation}
		for $i\in\{2,...,k\}$ and for $i=k+1$ if $u$ is a bound state.    The proof of   both cases is the same.
		\vskip 0.08in
		Since $uv$,  $M(r)$, $ -uu'\geq0$ in $(c_{i-1},\tau_i]$, we see from \eqref{defi of deriQMT} that  
		\begin{equation}\label{oi2}
			T_2'(r)=-\frac{2u'}{u}M(r)+2r^{n-1}uv\geq0,\quad \forall ~r\in[c_{i-1},\tau_i],
		\end{equation}
		and then we have 
		\begin{equation} \label{oi1}
			T_2(r)\geq T_2(c_{i-1})>0, \quad \forall~ r\in[c_{i-1},\tau_i].
		\end{equation}  This proves \eqref{sq1-1} in  the case $\tau_i\geq b_i$.  
		Then we consider the case $\tau_i<b_i$. We claim that
		\begin{equation}\label{sq1-2}
			\tau_i<b_i~~~~~\Rightarrow~~~~~uv<0~~~\text{in}~~~(\tau_i,b_i].
		\end{equation}
		This implies that $v$ has no zero in $(\tau_i,b_i]$.
		Suppose that \eqref{sq1-2} does not hold.  Then there exists $\tilde\tau_i\in(\tau_i,b_i]$ such that   
		\begin{equation}\label{sq1-3}  v(\tau_i)=v(\tilde{\tau}_i)=0, ~~~~~uv<0~~~\text{in}~~~(\tau_i,\tilde{\tau}_i).
		\end{equation}
		With the definitions of $T_2,~M,~Q$, we have 
		\begin{equation}\label{sq1-4}
			\begin{aligned}
				T_2(\tau_i)&=Q(\tau_i)-g_2(u(\tau_i))M(\tau_i)\\
				&=\tau_i^nu'v'+\tau_{i}^{n-1}g_2(u(\tau_i))uv'\\
				&=\tau_i^{n-1}u(\tau_i)v'(\tau_i)[g_2(u(\tau_i)-\omega(\tau_i))].
			\end{aligned}
		\end{equation}
		The same expression holds for $T_2(\tilde\tau_i)$ with $\tau_i$ replaced by $\tilde\tau_i$. It follows from \eqref{oi1} that  $T_2(\tau_i) > 0$.   Since $u(\tau_i)v'(\tau_i)<0$, we deduce from  \eqref{sq1-4} that $g_2(u(\tau_i))<\omega(\tau_i)$. Note that $g_2(u(r))=u\log u^2-1$ decreases in $(\tau_i,\tilde{\tau_i})$,  while $\omega(r)$ is increasing   in this interval by Lemma \ref{lemma positivefunction}, we have $$g_2(u(\tilde{\tau}_i))<\omega(\tilde{\tau}_i).$$ Because $v'$  has opposite signs at consecutive zeros,   we have $u(\tilde{\tau}_i)v'(\tilde{\tau}_i)>0$, which implies $T_2(\tilde{\tau}_i)<0$. Then from \eqref{oi2} and \eqref{oi1}, $T_2$ has a critical point within the interval $(\tau_i,\tilde\tau_i)$. 
		
		\vskip0.08in 
		
		Denote by $\tau_i^*$ the first critical point of $T_2$ in $(\tau_i,\tilde\tau_i)$. Then 
		\begin{equation}\label{sq1-5}
			\tau^*_i\in(\tau_i,\tilde{\tau}_i),~~~T_2(r),~T_2'(r)>0~\text{in}~[c_{i-1},\tau^*_i),~~~T_2'(\tau^*_i)=0.
		\end{equation}
		Since $u'(\tau_i^*)v(\tau_i^*)>0$ and $T_2'(\tau_i^*)=0$, we deduce from the definition of $T_2'$ in \eqref{defi of deriQMT} that 
		\begin{equation}\label{sq1-6}
			M(\tau_i^*)=\frac{\tau_i^{*n-1}u^2v}{u'}>0.
		\end{equation}
		On the other hand, we see from  \eqref{ConnectionPQ},  \eqref{sq1-6} that
		\begin{equation}\label{sq1-7}
			Q(\tau_i^*)=P(\tau_i^*)\cdot\frac{v(\tau_i^*)}{u(\tau_i^*)}<0,
		\end{equation}
		where the last inequality follows from the facts $u(\tau_i^*)v(\tau_i^*)<0$ and $P(r)>0$. Hence, we deduce from   \eqref{sq1-5} that   $T_2(\tau_i^*)>0$ and therefore, 
		\begin{equation}\label{sq1-8}
			g_2(u(\tau_i^*))M(\tau_i^*)=Q(\tau_i^*)-T_2(\tau_i^*)<0.
		\end{equation}
		Since $M(\tau_i^*)>0$, we get
		\begin{equation}\label{sq1-9}
			F(u(\tau_i^*))=\frac{1}{2}u^2(\tau_i^*)g_2(u(\tau_i^*))<0.
		\end{equation}
		However,  we have $F(u(\tau_i^*))>0$ because $\tau_i^*\in(\tau_i,b_i)$. This is a contradiction and we confirm \eqref{sq1-2}. 
		\vskip0.08in 
		Now we prove \eqref{sq1-1} by contradiction.  Suppose there exists a critical point $\tau_i^{**}$ of $T_2$ such that 
		\begin{equation}\label{sq1-10}
			\tau_i^{**}\in(\tau_i,b_i], ~~~~~T_2,~T_2'>0\text{ in } [c_{i-1},\tau_i^{**}),~~~~~T_2'(\tau_i^{**})=0.    
		\end{equation}
		Since  $u'(\tau_i^{**})v(\tau_i^{**})>0$,  by a similar argument as that of \eqref{sq1-6}, \eqref{sq1-7} and \eqref{sq1-8}, we obtain
		\begin{equation}\label{sq1-11}
			M(\tau_i^{**})=\frac{\tau_i^{**n-1}u^2v}{u'}>0,~~~~~Q(\tau_i^{**})= P(\tau_i^{**})\cdot\frac{v(\tau_i^{**})}{u(\tau_i^{**})}<0,
		\end{equation}
		and consequently
		\begin{equation}
			g_2(u(\tau_i^{**}))M(\tau_i^{**})=Q(\tau_i^{**})-T_2(\tau_i^{**})<0, 
		\end{equation}
		which again gives $g_2(u(\tau_i^{**}))<0$, and therefore $F(u(\tau_i^{**}))<0$. This contradicts the fact that $\tau_i^{**}\in(\tau_i,b_i]$.  Hence \eqref{sq1-1} holds.
		
		\vskip 0.1in
		{\it Step 2.} We   prove that for $i\in\{2,...,k\}$ and for $i=k+1$ if $u$ is a bound state,  
		\begin{equation}\label{sq2-1}
			Q(r),~Q_1(r),~Q_2(r)>0,\quad r\in[c_{i-1},\max\{b_i,\tau_i\}].
		\end{equation}  
		From \eqref{sq0}, we have
		\begin{equation}
			Q_1(c_{i-1})=Q_2(c_{i-1})=c_{i-1}^{n}f(u(c_{i-1}))v(c_{i-1})=Q(c_{i-1})>0.
		\end{equation}
		Moreover, by \eqref{defi of deriQMT} and Lemma~\ref{lemma sign}, the functions $Q$, $Q_1$, $Q_2$ are strictly increasing on $(c_{i-1},\tau_i)$. Consequently,  \begin{equation}
			Q(r),\,Q_1(r),\,Q_2(r) > 0\quad\text{for all } r\in[c_{i-1},\tau_i].
		\end{equation}
		This establishes \eqref{sq2-1} in the case $b_i\le\tau_i$.
		
		\vskip0.04in 
		Now assume $b_i > \tau_i$. Then \eqref{sq1-2} holds, so in $(\tau_i,b_i]$ we have $f(u)v < 0$, $Q'(r)<0$, and $u'v > 0$.   From \eqref{sq1-1} we obtain
		\begin{equation}
			Q(b_i)=T_2(b_i)>0.
		\end{equation}
		Since $Q(r)$ increases over $[c_{i-1},\tau_i]$ and decreases over $[\tau_i,b_i]$, it follows that $Q_2>Q_1>Q>0$ in $(\tau_i,b_i]$. Hence \eqref{sq2-1} holds also when $b_i>\tau_i$, completing the proof.
		
		\vskip 0.1in
		{\it Step 3.}  We prove the following:
		\begin{equation}\label{sq3-1}
			\left\{
			\begin{aligned}
				& Q_1>0 \quad \text{ in }~[\tau_i,\bar{b}_{i}] ~\text{ if }~ i\in\{2,...,k-1\};\\
				& Q_1>0 \quad\text{ in }~[\tau_{k},z_{k}] ~\text{ if }~ i=k \text{ and u is not a bound state};\\
				& Q_1>0 \quad\text{ in }~[\tau_k,\bar{b}_k] ~\text{ if }~ i=k \text{ and u is a bound state};\\
				& Q_1>0 \quad\text{ in }~[\tau_{k+1},\infty) ~\text{ if }~ i=k+1 \text{ and u is a bound state}.
			\end{aligned}
			\right.
		\end{equation}
		We argue by contradiction.   Let $q_i$ be the first zero of $Q_1$ in the respective interval. From \eqref{sq2-1}, we have $q_i > \max\{b_i,\tau_i\}$, and
		\begin{equation}\label{sq3-2}
			u'v>0 \text{ in }(\tau_i,q_i],~~~Q_n(r)\geq Q_1(r)>0\text{ in }[\tau_i,q_i),~~~Q_1(q_i)=0.
		\end{equation}We distinguish two cases. 
		First, we consider the case $b_i\leq\tau_i$. By   \eqref{defi of Ba}  and \eqref{sq3-2}, we have $B_0'(r) > 0$ on $(\tau_i,q_i)$. Using \eqref{sq2-1} we obtain
		\begin{equation}\label{sq3-3}
			B_0(q_i)>B_0(\tau_i)=Q(\tau_i)>0.
		\end{equation}
		Then we consider the case   $b_i>\tau_i$.  Then we see from \eqref{defi of Ba}  and \eqref{sq3-2} that  $B_0'(r)>0$ in $(b_i,q_i)$, and  again \eqref{sq2-1} yields
		\begin{equation}
			B_0(q_i)>B_0(b_i)=Q(b_i)>0.
		\end{equation}
		In both cases we conclude $B_0(q_i) > 0$. However,  similar to \eqref{oi3}, we have 
		\begin{equation}
			B_0(q_i)=-q_i^{n-1}u'v-2F(u)\frac{q_i^{n-1}v}{u'}=-\frac{2q_i^{n-1}v}{u'}E(q_i)<0.
		\end{equation}
		This is a contradiction and \eqref{sq3-1} is proved.
		
		\vskip 0.1in
		{\it Step 4.} We now complete the proof of  Proposition \ref{prop phasei}.
		We first claim that for $i\in\{2,...,k-1\}$ and for $i=k$ if $u$ is a bound state,
		\begin{equation}\label{sq4-1}
			Q(\bar{b}_i)>Q(b_i)>0.
		\end{equation}
		Indeed,   since $Q_1<0$ at the first possible zero of $v$ in the corresponding intervals in \eqref{sq3-1}, we see from \eqref{sq3-1}   that   $v$ has no zero in $(\tau_i,\bar{b}_i]$. Moreover $u'v>0$ in $(\tau_i,\bar{b}_i]$.   If $b_i \ge \tau_i$, then on $(b_i,\bar{b}_i]$ we have $Q_n(r) > Q_1(r) > 0$ and $B_0'(r)>0$, hence $$Q(\bar{b}_i)=B_0(\bar{b}_i) > B_0(b_i)=Q(b_i).$$  
		If $b_i < \tau_i$, then $B_0'(r)>0$ on $(\tau_i,\bar{b}_i)$, so $Q(\bar{b}_i) > Q(\tau_i)$. Since $Q'(r)>0$ on $(b_i,\tau_i)$, we obtain $$Q(\bar{b}_i) > Q(\tau_i) > Q(b_i) > 0,$$ where the positivity follows from \eqref{sq2-1}. This establishes \eqref{sq4-1}.
		
		\vskip0.08in
		Now consider $i\in\{2,\dots,k-1\}$ or $i=k$  if $u$
		is a bound state. Assume that $Q,~M,~T_2>0$ at $c_{i-1}$. By Lemma~\ref{lemma sign}, the first zero of $v$ in $[c_{i-1},c_i]$ occurs at $\tau_i\in(c_{i-1},r_i)$.   From \eqref{sq3-1}, $v$ has no zero in $(\tau_i,\bar{b}_i]$.  To prove that the zero is unique in $[c_{i-1},c_i]$, it suffices to show that $v$ has no zero in $[\bar{b}_i,c_i]$.
		We argue by contradiction.  
		Suppose there exists $\tilde\tau_i\in[\bar{b}_i,c_i]$ such that 
		\begin{equation}\label{sq4-2}
			u'v>0~\text{in}~(\tau_i,\tilde{\tau}_i),\qquad v(\tilde{\tau}_i)=0.
		\end{equation}
		Since $u'v'>0$ at $\tau_i$, we have $u'v'<0$ at $\tilde{\tau}_i$. Consequently,   $Q(\tilde{\tau}_i)=\tilde{\tau}_i^nu'v'<0$. On the other hand, on $[\bar{b}_i,\tilde\tau_i]$ we have $f(u)v>0$ and $Q'(r)=2r^{\,n-1}f(u)v>0$, so $Q(\tilde\tau_i) > Q(\bar{b}_i)$. By \eqref{sq4-1}, we obtain $Q(\bar{b}_i)>0$, hence $Q(\tilde\tau_i)>0$, a contradiction. Therefore $v$ has exactly one zero in $[c_{i-1},c_i]$.
		
		\vskip0.08in
		It remains to verify $Q,~M,~T_2>0$ at $c_i$. Because the zero of $v$ on $[c_{i-1},c_i]$ is unique, we have
		\begin{equation}\label{sq4-3}
			Q'(r)=2r^{\,n-1}f(u)v>0~ \text{on } ~[\bar{b}_i,c_i],\qquad u'v>0 ~\text{in } ~(\tau_i,c_i],\qquad Q(r)>0  ~\text{on }~ [\bar{b}_i,c_i],
		\end{equation}
		where the last inequality follows from \eqref{sq4-1}. 
		From \eqref{sq2-1} and \eqref{sq3-1}, we have $Q_1>0$ on $[c_{i-1},\bar{b}_i]$.  Moreover, we see from  \eqref{sq4-3} that  $Q_1(r)=Q(r)+r^{n-1}u'v>Q(r)>0$ on $[\bar{b}_i,c_i]$. Hence $Q_1>0$ on the whole interval $[c_{i-1},c_i]$.  On the other hand,  by \eqref{sq2-1}, we obtain $Q_2>0$ on $[c_{i-1},\tau_i]$.  Since $u'v>0$ in $(\tau_i,c_i]$, we have $Q_2 > Q_1 > 0$ on $(\tau_i,c_i]$, so $Q_2>0$ on $[c_{i-1},c_i]$ as well.  
		Using \eqref{sq1-1} and \eqref{sq4-1}, the hypotheses of Lemma~\ref{lemma T2} are satisfied on $[\bar{b}_i,c_i]$, giving $T_2>0$ on $[\bar{b}_i,c_i]$.   Finally, since  $M'(r)=2r^{\,n-1}uv$, it increases from $M(c_{i-1})>0$ to $M(\tau_i)$, then decreases to $M(z_i)=z_i^{\,n-1}u'v>0$, and finally increases to $M(c_i)$. Thus $M>0$ on $[c_{i-1},c_i]$. This completes the proof of Proposition~\ref{prop phasei} for the considered indices.

		\vskip0.08in
		The remaining cases are $i=k$ if $u$ is not a bound state, and $i=k+1$ if $u$ is a bound state. In both situations, the uniqueness of the zero of $v$ on $I_i$ follows directly from \eqref{sq3-1}, and the same arguments as above apply. The proof of  Proposition \ref{prop phasei} is finished. 
		
	\end{proof}

	\subsection{Proof of Theorems \ref{Thm} and \ref{theoremmain1}}
	In this subsection we prove our main results, namely Theorems~\ref{Thm} and \ref{theoremmain1}.  The proof relies on the following proposition concerning the behavior of the variation function $v$.
	\begin{proposition}\label{prop v}
		Let $u(r)=u(r,\alpha)$ be the solution of \eqref{radialform}. Then for any $\alpha>\alpha_*$, the  variation function $v(r)=v(r,\alpha)$ satisfies the following statements: 
		\begin{itemize}
			\medbreak
			\item[(i)]   If $n\geq 2$ and  $u(r)$ is a ground state,  then $v(r)$ has exactly one zero $\tau_1\in(0,+\infty)$.  Moreover, $v(r)$ is strictly monotonically decreasing for sufficiently large $r$ and $$\lim\limits_{r\to\infty}v(r)=-\infty.$$
			\medbreak
			\item[(ii)]   If $n\geq 3$ and  $u(r)$ is a $k$-node solution  with zeros $z_1<\cdots<z_k$, then $v(r)$ has precisely $k$ zeros $\tau_1<\cdots<\tau_k$ on $[0,z_k]$, where $\tau_1\in(0,z_1)$, and $\tau_i\in(z_{i-1},z_i)$ for $2\leq i\leq k$.   If, in addition, $u(r)$ is a bound state, then $v(r)$ has precisely one additional zero $\tau_{k+1}  \in [z_k,+\infty)$ satisfying  $\tau_{k+1}>c_k>z_k$, and $v(r)$ is strictly monotone for sufficiently large $r$ with  $$\lim\limits_{r\to\infty}|v(r)|=+\infty.$$
		\end{itemize} 
		
	\end{proposition}
	
	\begin{proof}
		Statement (i) follows directly from Proposition \ref{prop vgroundstate}. We now verify the  statement (ii).    Assume that  $n\geq 3$ and $u(r)$ is a $k$-node solution  with zeros $z_1<\cdots<z_k$.  When $k=1$, it follows from Proposition \ref{prop phase1 one zero}  that $v(r)$ has exactly one zero  $\tau_1$ on $[0,z_k]$. When $k=2$, Proposition \ref{prop phase1} shows that $v(r)$ has a unique zero $\tau_1$ in $[0,c_1]$, and $\tau_1\in \sbr{0,r_1}$. By Proposition \ref{prop phasei},  we have $Q(c_1),~M(c_1),~T_2(c_1)>0$, and $v$ has a unique zero $\tau_2$ in $[c_1,z_2]$ with $\tau_2\in(c_1,r_2)$.   Since
		$r_1 < z_1 < c_1 < r_2 < z_2$, it follows that $\tau_2$ is the unique zero of $v$ on $[z_1,z_2]$.  For $k\ge 3$, the conclusion follows from Proposition \ref{prop phasei} by induction on $k$.
		\vskip0.08in
		Now we assume, in addition,  that $u(r)$ is a $k$-node bound state. By Lemma \ref{lemma nodalcriticalpoint}, there exists a unique critical point $c_k>z_k$. Moreover, it follows from Proposition \ref{prop phasei} that $Q(c_k),~M(c_k),~T_2(c_k)>0$ and   $v(r)$ has a unique zero $\tau_{k+1}$ on $[c_k,+\infty)$ with  $\tau_{k+1}\in(c_k,r_{k+1})$. Since
		$\tau_{k}<r_{k} < z_{k} < c_k < r_{k+1} $, we deduce that $\tau_{k+1}$ is the unique zero of $v$ on $[z_k,+\infty)$.  Finally, we prove that $v(r)$ is strictly monotone for sufficiently large $r$ with  $$\lim\limits_{r\to\infty}|v(r)|=+\infty.$$
		By \eqref{sq2-1} and \eqref{sq3-1}, we have $Q_1(r)>0$ for all $r\ge c_k$. First we prove that $\lim_{r\to\infty}B_0(r)$ is positive (possibly $+\infty$).
		If $\tau_{k+1}\ge b_{k+1}$, then
		\begin{equation}\label{mq1}
			0<|u(r)|<\alpha_*,\quad u'v>0,\quad\forall r\in(\tau_{k+1},\infty),
		\end{equation}
		and \eqref{mq1} implies that
		\begin{equation}
			F(u(r))<0,\quad Q_n(r)\ge Q_1(r)>0,\quad\forall r\in(\tau_{k+1},\infty).
		\end{equation}
		Hence $B_0'(r)>0$ in $(\tau_{k+1},\infty)$ and
		\begin{equation}
			\lim_{r\to\infty} B_0(r) > B_0(\tau_{k+1}) = Q(\tau_{k+1}) > 0.
		\end{equation}
		If $\tau_{k+1}<b_{k+1}$, then we similarly have
		\begin{equation}
			\lim_{r\to\infty} B_0(r) > B_0(b_{k+1}) = Q(b_{k+1}) > 0.
		\end{equation}
		Since $\lim_{r\to\infty}u(r)=0$, we have $f'(u(r))<0$ for sufficiently large $r$. Because $(r^{n-1}v')' = -r^{n-1}f'(u)v$ and $v$ does not change sign for large $r$, the right-hand side is eventually of one sign; hence $r^{n-1}v'$ is eventually monotone. Consequently, $v'$ is eventually of constant sign, so $v$ is eventually monotone. Suppose that $v$ has a finite limit. Similar to the proof of \eqref{lim_of_Q}, we have
		\[
		Q(r)=r^n[u'v'+f(u)v]+(n-2)r^{n-1}u'v \to 0 \quad \text{as } r\to\infty.
		\]
		Then
		\[
		B_0(r)=Q(r)-2F(u)\frac{r^{n-1}v}{u'}=Q(r)-r^{n-2}u(\log u^2-1)\frac{ru}{u'}v \to 0,
		\]
		which contradicts $\lim_{r\to\infty}B_0(r)>0$. Hence $\lim_{r\to\infty}|v(r)| = \infty$.  
	\end{proof}
	
	With the help of Proposition \ref{prop v}, we proceed to complete the proof of Theorems \ref{Thm} and \ref{theoremmain1}. Since Theorem \ref{Thm} follows directly from Theorem \ref{theoremmain1}, it suffices to prove the latter.  In what follows,    $\mathcal{N}(\alpha)$ denotes the number of zeros of   $u(r,\alpha)$ over $(0,\infty)$. 
	By Lemmas  \ref{lemma existcriticalpoint} and \ref{lemma nodalcriticalpoint}, we have $\mathcal{N}(\alpha)=0$ for all $\alpha\leq\alpha_*$, and $\mathcal{N}(\alpha)$ is finite for each $\alpha>\alpha_*$.
	
	\begin{proof}[\bf Proof for parts (i) and (iii) of Theorem \ref{theoremmain1}] 
		Let $n\geq 2$ and  $\bar \alpha>\alpha_*$ be such that $\bar{u}(r)=u(r,\bar\alpha)$ is a ground state. Then $\mathcal{N}(\bar\alpha)=0$. We claim that
		\begin{equation}\label{ik1}
			\mathcal{N}(\alpha)=0 ~~\text{ for all}~~ \alpha\leq \bar\alpha, \quad \text{ and }\quad  \liminf\limits_{\alpha\to \bar\alpha^+}\mathcal{N}(\alpha)> \mathcal{N}(\bar\alpha)=0.
		\end{equation}
		To prove the first statement,   assume that there exists  $\tilde\alpha<\bar\alpha$ such that  $\mathcal{N}(\tilde\alpha)\geq1$.  Denote by $z(\tilde\alpha)$  the  zero of $u(r,\tilde\alpha)$. Note that $u, v\in C^1(0,\infty)$. By the Implicit Function Theorem, we  differentiate the identity $u(z(\tilde\alpha),\tilde\alpha)=0$ at $\alpha=\tilde\alpha$  and obtain
		\begin{equation}\label{derivative-zero}  u'(z(\tilde\alpha),\tilde\alpha)z'(\tilde\alpha)
			+v(z(\tilde\alpha),\tilde\alpha)=0.
		\end{equation}
		By   Proposition  \ref{prop v}, we have $u'v>0$  at $z_1(\tilde\alpha)$. Consequently, 
		\begin{equation}\label{new-mono1-ground}
			z'(\tilde\alpha)=-\frac{v(z(\tilde\alpha),\tilde\alpha)}{u'(z(\tilde\alpha),\tilde\alpha)}<0.
		\end{equation}
		Hence the   zero moves left as  $\alpha$ increases.  By continuous dependence on $\alpha$, the solution $u(r,\alpha)$ must have at least one zero for every $\alpha>\tilde\alpha$.   This   contradicts  the fact  $\mathcal{N}(\bar\alpha)=0$. Therefore, for all $\alpha\leq \bar\alpha$, we have $ \mathcal{N}(\alpha) = 0 $. This proves the first part of \eqref{ik1}.
		\vskip 0.08in
		We proceed to prove    the second statement in \eqref{ik1} by following the McLeod's arguments in \cite{M1993}.  By Lemma \ref{lemma existcriticalpoint}, $\bar{u}$ is positive and decreases to zero as $r\to+\infty$. Let $\bar{v}=v(r,\bar\alpha)$ be the variation  corresponding to $\bar{u}$. By Proposition \ref{prop v}, we have  $\bar{v},~\bar{v}'<0$ for $r$ sufficiently large and $\lim_{r\to+\infty}\bar{v}(r)=-\infty$. Choose $\bar{r}$ large enough so that
		\begin{equation}\label{new-jump-ground1}
			\bar v,~\bar{v}'<0,\quad f'(\bar{u}(r))<0,\quad\forall r\in[\bar{r},+\infty).
		\end{equation}
		Then by Proposition \ref{prop v} again,  $\bar{r}$ is behind   the last critical point of $\bar{u}$, and $\bar{u}$ is strictly decreasing on  $(\bar{r},+\infty)$. 
		For any  $\alpha>\bar\alpha$, set $w(r):=u(r)-\bar{u}(r)$. If   $\alpha$ is sufficiently close to $\bar{\alpha}$, then by continuity we have 
		\begin{equation}\label{new-jump-ground2}
			w(\bar{r})=u(\bar{r},\alpha)-u(\bar{r},\bar\alpha)=\int_{\bar{\alpha}}^\alpha v(\bar{r},t)~\mathrm{d}t<0,
		\end{equation}
		and
		\begin{equation}\label{new-jump-ground3}
			w'(\bar{r})=u'(\bar{r},\alpha)-u'(\bar{r},\bar\alpha)=\int_{\bar{\alpha}}^\alpha v'(\bar{r},t)~\mathrm{d}t<0.  
		\end{equation}
		Also, by continuity and the fact that $\bar u(\bar{r})>0$ on $\mbr{0,\bar r}$, we have $u(\bar{r})>0$ for $\alpha$ sufficiently close to $\bar\alpha$ and $u$ has no zero in $(0,\bar{r})$. Suppose, for contradiction, that  $u$ has no zero in $(\bar{r},+\infty)$.   Then $u>0$ in $(\bar{r},+\infty)$ because $u(\bar{r})>0$. By Lemma \ref{lemma oscalliate} and \ref{lemma positivenodalsolution}, either $u$ tends to zero or $u$ oscillates about 1 in $(\bar{r},+\infty)$.  Since  $w(\bar{r})<0$ and $w'(\bar{r})<0$, the function $w$ attains a negative minimum at some $\hat{r}\in(\bar{r},+\infty)$. Hence, 
		\begin{equation}
			w(\hat{r})<0, \quad w'(\hat{r})=0, \quad w''(\hat{r})\ge 0.
		\end{equation}
		Because $\bar{u}$ is decreasing in $(\bar{r},+\infty)$ and $\bar{u}(\hat{r})>u(\hat{r})>0$, we have
		\begin{equation}\label{new-jump-ground4}
			f'(u(\hat{r}))<0,\quad f'(\bar{u}(\hat{r}))<0.
		\end{equation}
		Both $u$ and $\bar{u}$ satisfy \eqref{radialform} with parameters $\alpha$ and $\bar\alpha$, respectively. Therefore, $w$ satisfies  
		\begin{equation}\label{new-jump-ground5}
			w''+\frac{n-1}{r}w'+f'(\tilde{u}(r))w=0,
		\end{equation}
		where $\tilde{u}(r)$ lies between  $u(r)$ and $\bar{u}(r)$.  Evaluating at $r=\hat{r}$ yields  $w''(\hat{r})=-f'(\tilde{u}(\hat{r}))w(\hat{r})<0$,  contradicting  $w''(\hat{r}) \geq 0$. Hence our supposition is false, and $u$ must have at least one zero in $\sbr{\bar r, \infty}$; thus $\mathcal{N}(\alpha)>0$  for all $\alpha$ sufficiently close to $\bar \alpha$ from above. Consequently, $\liminf_{\alpha\to \bar\alpha^+}\mathcal{N}(\alpha)>0$,  which proves the second part of \eqref{ik1}.
		
		\vskip 0.08in
		We now complete the proof of parts (i) and (iii) of Theorem \ref{theoremmain1}.   Bialynicki-Birula and Mycielski \cite{Bialynicki1976,Bialynicki1979} established the existence of a ground state
		\begin{equation}
			u_0(r)= \exp\sbr{-\frac{r^2}{2}+\frac{n}
				{2}}  , \quad 0\leq r <\infty,
		\end{equation}
		with      $\alpha_0=u_0(0)=e^{n/2}$. Then by \eqref{ik1} we have 
		\begin{equation} \label{iqq1}
			\mathcal{N}(\alpha)=0 ~~\text{ for all}~~ \alpha\leq  \alpha_0, \quad \text{ and }\quad  \liminf\limits_{\alpha\to  \alpha_0^+}\mathcal{N}(\alpha)>0.
		\end{equation}
		We  now prove that $u_0$ is the unique ground state of \eqref{radialform} for $n\geq 2$ by investigating the behavior of $\mathcal{N}(\alpha)$; this argument differs from the uniqueness proof for the ground state given in  \cite{avenia2014,TroyARMA2018}.  Let  $  \hat \alpha>\alpha_*$ be such that $ \hat{u}(r)=u(r,\hat\alpha)$ is another ground state with $\hat \alpha \neq \alpha_0$. Then we deduce from \eqref{ik1}   that 
		\begin{equation}\label{iqq2}
			\mathcal{N}(\alpha)=0 ~~\text{ for all}~~ \alpha\leq  \hat\alpha, \quad \text{ and }\quad  \liminf\limits_{\alpha\to  \hat\alpha^+}\mathcal{N}(\alpha)>0.
		\end{equation}
		If $\hat\alpha<\alpha_0$, then from \eqref{iqq1} we see that $\liminf_{\alpha\to \hat \alpha^+}\mathcal{N}(\alpha)= 0$,
		which   contradicts  \eqref{iqq2}. Similarly, if $\hat\alpha > \alpha_0$, then from \eqref{iqq2} we obtain   $\liminf_{\alpha\to   \alpha_0^+}\mathcal{N}(\alpha)= 0$, contradicting \eqref{iqq1}.
		Hence $u_0(r)$ is the unique ground state and part (i) of Theorem \ref{theoremmain1} is proved. 
		\vskip0.08in
		Finally, if $\alpha=1$, then the uniqueness of the solution to the ODE  implies that $u\equiv1$.  For    any $\alpha<\alpha_0$ with $\alpha\neq 1$, we see from \eqref{iqq1} that  $ \mathcal{N}(\alpha)=0$ and  that the function $u(r,\alpha)$  is a positive solution to \eqref{radialform}. Such a solution cannot be a ground state  since the unique ground state corresponds to $\alpha_0$.  Consequently, by Lemmas \ref{lemma positivenodalsolution} and  \ref{lemma oscalliate}, the function  $u$ oscillates about   $u\equiv  1$ and $\inf u>0$. Hence part (iii) of Theorem \ref{theoremmain1} is confirmed. 
	\end{proof}
	
	\begin{proof}[\bf Proof for parts (ii) and (iv) of Theorem \ref{theoremmain1}]
		For $n\geq 2$ and any $k\geq 1$, let $\alpha_k>\alpha_0$ be such that $u_k(r)=u(r,\alpha_k)$ is a $k$-node bound state with zeros $0<z_1<z_2<\cdots<z_k<\infty$. Then $\mathcal{N}(\alpha_k)=k$. The existence of such $\alpha_k$  follows from the fact that \eqref{radialform} admits bound states with any prescribed number of zeros (see Shuai \cite{W.ShuaiNonlinearity2019}). 
		By Lemma \ref{lemma nodalcriticalpoint}, for $1\leq i \leq k-1$  there exists   exactly one critical point $c_i\in \sbr{z_i,z_{i+1}}$  and one critical point  $c_k$ in $(z_k, \infty)$ beyond which $|u_k(r)|$ decreases strictly.   At each critical point, $|u_k|>\alpha_*=e^{1/2}$. The asymptotic properties \eqref{decay1} and \eqref{decay2} follow  directly from  Lemma \ref{lemma decay}.

		\vskip0.08in
		When $n=2$, Cort\'azar,   Garc\'ia-Huidobro and  Yarur \cite[Theorem 1.1]{Cortazar2011} proved that the sequence   $\{\alpha_k\}_{k=1}^\infty$  satisfies  \eqref{increasing prop}
		and that  $u_k$ is the unique bound state of \eqref{radialform} with 
		precisely $k$ zeros. Moreover, they showed that for any 
		$\alpha \in (\alpha_k,\alpha_{k+1})$, 
		$u(r)=u(r,\alpha)$ is a nodal solution with exactly $k+1$ zeros. 
		By Lemma~\ref{lemma positivenodalsolution}, such a solution oscillates 
		about $u\equiv 1$ or $u\equiv -1$ after its last zero. 
		
		\vskip0.08in
		It remains to consider the case $n\geq3$.  Let $\bar \alpha>\alpha_0$ be such that $\bar{u}(r)=u(r,\bar\alpha)$ is a nodal solution  with exactly $k$ zeros  $z_1(\bar\alpha)<\cdots<z_k(\bar\alpha)$. 
		We claim that
		\begin{equation}\label{new monotonicity prop}
			\mathcal{N}(\alpha)\geq\mathcal{N}(\bar\alpha),\quad\forall \alpha\geq\bar\alpha.
		\end{equation}  
		For  $\alpha>\bar\alpha$, let $z(\alpha)$ be any zero of $u(r)=u(r,\alpha)$. Still by the Implicit Function Theorem, differentiating $u(z(\alpha),\alpha)=0$ gives 
		\begin{equation}
			u'(z(\alpha),\alpha)z'(\alpha)+v(z(\alpha),\alpha)=0.
		\end{equation}
		Since $u'v>0$ wherever $u=0$ by Proposition \ref{prop v}, we obtain
		\begin{equation}\label{new-mono1}
			z'(\alpha)=-\frac{v(z(\alpha),\alpha)}{u'(z(\alpha),\alpha)}<0.
		\end{equation}
		Thus each zero moves strictly leftward as $\alpha$ increases. In particular, for $\alpha$  slightly larger than $\bar\alpha$,   the solution $u(r,\alpha)$ still has at least $k$ zeros $z_1(\alpha)<z_1(\bar \alpha)$, $\ldots$ ,$z_k(\alpha)<z_k(\bar \alpha)$.   Hence $\mathcal{N}(\alpha)\geq\mathcal{N}(\bar{\alpha})$,  and \eqref{new monotonicity prop} holds.

		\vskip 0.08in
		By Lemma \ref{lemma positivenodalsolution},  the nodal solution $\bar{u}(r)=u(r,\bar\alpha)$ is either a bound state or an oscillatory solution. If $\bar{u}$ is  a  bound state, then we see from Proposition \ref{prop v} that  $\bar v(r)=v(r,\bar \alpha)$ is strictly monotone for sufficiently large $r$ with  $\lim_{r\to\infty}|\bar v(r)|=+\infty$.  By a similar argument, with slight modifications, to the one used in the proof of \eqref{ik1},  we have 
		\begin{equation}\label{new-jump-prop}
			\liminf\limits_{\alpha\to\bar\alpha^+}\mathcal{N}(\alpha)>\mathcal{N}(\bar\alpha).
		\end{equation}
		On the other hand,    if $\bar{u}$ is  an oscillatory solution, then $\mathcal{N}(\alpha)$ remains  constant in a neighborhood of $\bar\alpha$.  That is, 
		\begin{equation}\label{new-pausing-property}
			\mathcal{N}(\alpha)=\mathcal{N}(\bar\alpha), \quad \text{ for  } ~\alpha ~\text{ sufficiently close to } ~\bar\alpha.
		\end{equation}
		Indeed, since $\bar{u}$ is an oscillatory solution, by Lemma \ref{lemma oscalliate} we can choose $\bar{r}$ to be a critical point of $\bar{u}$ such that $|\bar{u}(\bar{r})|<1$. By Lemma \ref{lemma nodalcriticalpoint}, $\bar{r}$ lies behind the last zero of $\bar{u}$, and $\bar{E}$ is negative at $\bar{r}$, where $\bar E$ is the energy functional \eqref{defi of E} of $\bar u$. By continuity, for $\alpha$ sufficiently close to $\bar\alpha$, the solution $u(r,\alpha)$ has the same number of zeros as $\bar{u}$ does in $(0,\bar{r}]$, and its energy $E$ is negative at $\bar{r}$ as well. Hence by Lemma \ref{lemma oscalliate}, $u$ has no zeros in $(\bar{r},+\infty)$ and thus $\mathcal{N}(\alpha)=\mathcal{N}(\bar\alpha)$.
		\vskip 0.08in
		With the help of \eqref{new monotonicity prop}, \eqref{new-jump-prop} and \eqref{new-pausing-property}, we have
		\begin{equation} \label{e11}
			\mathcal{N}(\alpha)\geq k+1,\quad\forall\alpha>\alpha_k;\quad\mathcal{N}(\alpha)\leq k,\quad\forall\alpha<\alpha_k.
		\end{equation}
		Consequently, $u_k$ is the unique $k$-node bound state, and   $\alpha_0<  \alpha_1<\alpha_2<\cdots  $. Because $ \mathcal{N}(\alpha)$ is finite for all $\alpha>0$, we have $\lim_{k\to \infty} \alpha_k=\infty$. For each  
		$\alpha \in (\alpha_k,\alpha_{k+1})$, we see from \eqref{e11} that  $ \mathcal{N}(\alpha)= k+1$; thus,
		$u(r)=u(r,\alpha)$ is a nodal solution with exactly $k+1$ zeros. By \eqref{new-jump-prop}, such solution cannot be a bound state. Lemmas \ref{lemma positivenodalsolution} and \ref{lemma nodalcriticalpoint} imply that  $u$ oscillates about $u\equiv1$ or $u\equiv-1$ behind its last zero. Thus we complete the proof of part (ii) and (iv) of Theorem \ref{theoremmain1}.     
	\end{proof}

	\medskip
	
	{\small \noindent \textbf{Acknowledgements:} This work is funded by National Key R\&D Program of China (Grant 2023YFA1010001) and NSFC(12171265).  Liu is supported by the National Funded Postdoctoral Researcher Program (GZB20240945)  and China Postdoctoral Science Foundation (2025M784442).
	}
	
	\medskip
	
	{\small \noindent \textbf{Statements and Declarations:} The authors have no relevant financial or non-financial interests to disclose.}
	
	\medskip
	
	{\small \noindent \textbf{Data availability:} Data sharing is not applicable to this article as no datasets were generated or analysed during the current study.}

\end{document}